\documentclass[a4paper, 11pt, english]{article}
\usepackage[utf8]{inputenc}
\usepackage[T1]{fontenc}
\usepackage{graphicx}
\usepackage{stmaryrd}
\usepackage[a4paper]{geometry}
\usepackage{amsmath,amsfonts,amssymb,amsthm,epsfig,epstopdf,url,array}
\usepackage{rotating}
\usepackage[colorlinks=true,citecolor=red,linkcolor=blue,pdfpagetransition=Blinds]{hyperref}
\usepackage{cleveref}
\usepackage{nameref}
\usepackage{enumitem}
\usepackage{comment}
\Crefname{paragraph}{Section}{Sections}
\usepackage{fancyhdr}

\usepackage{fullpage}

\newcommand{\ensemblenombre}[1]{\mathbb{#1}}
\newcommand{\N}{\ensemblenombre{N}}

\newcommand{\R}{} % Probleme LateXML
\renewcommand{\R}{\ensemblenombre{R}}

\newcommand{\norme}[1]{\left\lVert#1\right\rVert}

\newcommand{\dive}[1]{\mathrm{div}}
\newcommand{\ov}[1]{\overline{#1}}
\providecommand{\keywords}[1]{\noindent {\textit{Keywords:}} #1}

\theoremstyle{plain} 
\newtheorem{prop}{Proposition}[section] 
\newtheorem{theo}[prop]{Theorem}

\newtheorem{lem}[prop]{Lemma}
\newtheorem{cor}[prop]{Corollary}
\theoremstyle{definition}

\newtheorem{rmk}[prop]{Remark}

\newtheorem{ass}[prop]{Assumption}

\def\dx{\,\textnormal{d}x}
\def\dt{\textnormal{d}t}
\def\d{\textnormal{d}}
\def\esp{{\mathbb{E}}}
\def\dom{\mathcal{D}}
\def\fil{\mathcal{F}}
\newcommand{\vertiii}[1]{{\left\vert\kern-0.25ex\left\vert\kern-0.25ex\left\vert #1 
    \right\vert\kern-0.25ex\right\vert\kern-0.25ex\right\vert}}

\makeatletter
\let\original@addcontentsline\addcontentsline
\newcommand{\dummy@addcontentsline}[3]{}
\newcommand{\DeactivateToc}{\let\addcontentsline\dummy@addcontentsline}
\newcommand{\ActivateToc}{\let\addcontentsline\original@addcontentsline}
\makeatother

\begin{document}

\title{Statistical null-controllability of stochastic nonlinear parabolic equations}
\author{V\'ictor Hern\'andez-Santamar\'ia\thanks{Instituto de Matem\'aticas, Universidad Nacional Aut\'onoma de M\'exico, Circuito Exterior, C.U., C.P. 04510 CDMX, Mexico. E-mail: \texttt{victor.santamaria@im.unam.mx}} \and K\'evin Le Balc'h\thanks{Institut de Math\'ematiques de Bordeaux, 351 Cours de la Lib\'eration, 33400 Bordeaux, France. E-mail: \texttt{kevin.le-balch@math.u-bordeaux.fr}} \and Liliana Peralta\thanks{Centro de Investigaci\'on en Matem\'aticas, UAEH, Carretera Pachuca-Tulancingo km 4.5 Pachuca, Hidalgo 42184, Mexico. E-mail: \texttt{liliana\_peralta@uaeh.edu.mx}}}

\maketitle

\begin{abstract}
In this paper, we consider forward stochastic nonlinear parabolic equations, with a control localized in the drift term. Under suitable assumptions, we prove the small-time global null-controllability, with a truncated nonlinearity. We also prove the “statistical” local null-controllability of the true system. The proof relies on a precise estimation of the cost of null-controllability of the stochastic heat equation and on an adaptation of the source term method to the stochastic setting. The main difficulty comes from the estimation of the nonlinearity in the fixed point argument due to the lack of regularity (in probability) of the functional spaces where stochastic parabolic equations are well-posed. This main issue is tackled through a truncation procedure. As relevant examples that are covered by our results, let us mention the stochastic Burgers equation in the one dimensional case and the Allen-Cahn equation up to the three-dimensional setting.
\end{abstract}
\footnotesize
\tableofcontents
\normalsize

\section{Introduction}

\subsection{Main result}
Let $T>0$ be a positive time, $\dom$ be a sufficiently smooth bounded, connected, open subset of $\R^n$, with $1 \leq n \leq 3$, whose boundary is denoted by $\Gamma := \partial\dom$ and $\dom_0$ be a nonempty open subset of $\dom$. We introduce the notation $\chi_{\dom_0}$, for the characteristic function of the set $\dom_0$.
 
Let $(\Omega, \mathcal{F}, \{\mathcal{F}_t\}_{t \geq 0}, \mathbb{P})$ be a complete filtered probability space on which a one-dimensional standard Brownian motion $\{ W(t)\}_{t \geq 0}$ is defined such that $\{\mathcal{F}_t\}_{t \geq 0}$ is the natural filtration generated by $W(\cdot)$ augmented by all the $\mathbb{P}$-null sets in $\mathcal{F}$. Let $X$ be a Banach space, for every $p \in [1, +\infty]$, we introduce
\begin{equation*}
L_{\mathcal{F}}^p(0,T;X) := \{ \phi : \phi\ \text{is an }X\text{-valued } \mathcal{F}_t\text{-adapted process on } [0,T]\ \text{and}\ \phi \in L^p([0,T] \times \Omega ;X)\},
\end{equation*}
endowed with the canonical norm and we denote by $L_{\mathcal{F}}^2(\Omega; C([0,T];X))$ the Banach space consisting on all $X$-valued $\mathcal{F}_t$-adapted process $\phi(\cdot)$ such that $\mathbb{E}\left(\norme{\phi(\cdot)}_{C([0,T];X)}^2 \right) < \infty$, also equipped with the canonical norm.

 We consider the stochastic semilinear heat equation

\begin{equation}\label{eq:heat_seminlinear}
\begin{cases}
\d y = (\Delta y + \chi_{\dom_0}h + f(y, \nabla y))\dt+(a y + g(y))\d W(t) &\text{in } (0,T)\times\dom, \\
y=0 &\text{on } (0,T)\times\Gamma, \\
y(0, \cdot)=y_0 &\text{in } \dom.
\end{cases}
\end{equation}
where $a \in \R$ and $f, g$ satisfy the following hypothesis.

\begin{ass}
There exist $\alpha, \beta, \gamma \in \R$ such that
\begin{align}
\notag
\forall (s,u) \in \R\times \R^n,\ f(s,u) &= \alpha s^p + \beta  s^q u, \quad p> 1,\  q \geq 1\ \text{for}\ n=1,\\
& = \alpha s^p, \quad p> 1\ \text{for}\ n=2,\label{eq:hypf}\\
& = \alpha s^p,\quad p \in (1,3]\ \text{for}\ n=3,\notag\\
\notag
g(s) &= \gamma s^r, \quad r> 1\ \text{for}\ n=1,\\
& =0,\ \text{for}\ n=2,\label{eq:hypg}\\
& =0,\ \text{for}\ n=3.\notag
\end{align}
\end{ass}

In the controlled system \eqref{eq:heat_seminlinear}, $y$ denotes the state while $h$ denotes the control, whose support is localized in $\dom_0$. We are interested in the null-controllability at time $T>0$ of \eqref{eq:heat_seminlinear}, that is to say we wonder if there exists a control $h$ such that the solution $y$ of \eqref{eq:heat_seminlinear} satisfies $y(T, \cdot) = 0$ in $\dom$, a.s.

Before stating the main results of the paper, let us introduce some notations. First, we define
\begin{equation*}
\forall t \in [0,T),\ \hat{\rho}(t) = \exp (-C/(T-t)),
\end{equation*}
where the constant $C>0$ will be defined later in the paper (see \Cref{sec:regular} below) and will only depend on $\dom$, $\dom_0$, $a$, $p$, $q$ and $r$.

We introduce the functional space: for every $t \in [0,T]$, 
\begin{align}
X_t :=\notag \Bigg\{ & y \in C([0,t];H_0^1(\dom)) \cap L^2(0,t;H^2(\dom))\ :\\
&   \sup_{0\leq s \leq t}\left\|\frac{y(s)}{\hat{\rho}(s)}\right\|_{H_0^1(\dom)} + \left(\int_{0}^{t}\left\|\frac{y(s)}{\hat{\rho}(s)}\right\|^2_{H^2(\dom)}\d{s}\right)^{1/2} < + \infty \Bigg\},\label{eq:defXt}
\end{align}
endowed with the corresponding norm. 

For each $R>0$, defining $\varphi_{R}\in C_0^\infty(\mathbb R^{+})$ such that
\begin{equation}
\label{eq:defvarphiR}
\varphi_R(s)=\begin{cases}
1, & s\leq R, \\
0, & s\geq 2R,
\end{cases}\quad
\text{and}\ \norme{\varphi_R'}_{\infty} \leq C /R,
\end{equation}
we introduce the truncated semilinearities $f_R$ and $g_R$ defined as follows
\begin{align}
\label{eq:deffR}
\forall (t, x, y) \in [0,T] \times \dom \times X_T,\ f_R(t, x, y) &= \varphi_R(\norme{y}_{X_t}) f(y(t,x), \nabla y (t,x)),\\
g_R(t,x,y) &= \varphi_R(\norme{y}_{X_t}) g(y(t,x)). \label{eq:defggR}
\end{align}
For convenience, from now we will abridge the notation in $f_R(y, \nabla y)$, $g_R(y)$ and we introduce the corresponding semilinear heat equation 
\begin{equation}\label{eq:heat_seminlinear_R}
\begin{cases}
\d y = (\Delta y + \chi_{\dom_0}h + f_R(y, \nabla y))\dt+(a y + g_R(y))\d W(t) &\text{in } (0,T)\times\dom, \\
y=0 &\text{on } (0,T)\times\Gamma, \\
y(0, \cdot)=y_0 &\text{in } \dom.
\end{cases}
\end{equation}

Now, we state our two main results.
\begin{theo}
\label{th:mainresult1}
Let $T>0$. There exists $R>0$ sufficiently small such that for every initial data $y_0 \in L^2(\Omega,\fil_0;H_0^1(\dom))$, there exists a control $h \in L^2_{\fil}(0,T;L^2(\dom))$, such that the solution $y$ of \eqref{eq:heat_seminlinear_R} satisfies $y(T, \cdot) = 0$ in $\dom$, a.s. Moreover, we have the following estimate
\begin{equation}
\label{eq:estimatenonlinearityPropfR}
 \esp \left(\norme{y}_{X_T}^2 \right)
\leq C^2\esp\left(\|y_0\|^2_{H_0^1(\dom)}\right),
\end{equation}
for a positive constant $C>0$ depending only on $T$, $\dom$, $\dom_0$, $a$, $\alpha$, $\beta$, $p$, $q$.
\end{theo}
\begin{theo}
\label{th:mainresult2}
Let $\epsilon >0$ and $T>0$ be given. Then, there exists $\delta=\delta(\epsilon)>0$ such that for every initial data $y_0 \in L^2(\Omega,\fil_0;H^1(\dom))$ verifying $\norme{y_0}_{L^2(\Omega,\fil_0;H_0^1(\dom))} \leq \delta$, there exists a control $h \in L^2_\fil(0,T;L^2(\dom_0))$ such that the solution $y$ of \eqref{eq:heat_seminlinear_R} satisfies $y(T, \cdot) = 0$ in $\dom$, a.s. and
\begin{equation}
\label{eq:probafRf}
\mathbb{P} \bigg( \left\{ f_R(y, \nabla y) = f(y, \nabla y) \right\} \cap   \left\{ g_R(y) = g(y) \right\} \bigg) \geq 1 - \varepsilon.
\end{equation}
\end{theo}

Before continuing, let us make some comments on \Cref{th:mainresult1} and \Cref{th:mainresult2}.
\begin{itemize}
\item \Cref{th:mainresult1} is a small-time global null-controllability result for the equation \eqref{eq:heat_seminlinear_R} which corresponds to \eqref{eq:heat_seminlinear}, with  truncated semilinearities $f_R$ and $g_R$. Remark that the parameter of truncation $R$ is taken sufficiently small then if the solution $y$ of \eqref{eq:heat_seminlinear_R} is too big in the space $X_T$, $f_R$ and $g_R$ vanish.

\item \Cref{th:mainresult2} is a “small-time \textit{statistical} local null-controllability” for the equation \eqref{eq:heat_seminlinear}. Indeed, we justify this new terminology as follows. Given any small time $T>0$ and a small constant $\epsilon>0$, we can find a ball of size $\delta>0$ such that, with a confidence level $1-\epsilon$, we can steer any initial data smaller than $\delta$ for system \eqref{eq:heat_seminlinear} to zero. One could compare \Cref{th:mainresult2} to the results obtained in \cite[Theorem 4.6]{GHV14} about global existence for stochastic Euler equations in the three dimensional case. Indeed, they prove that for every $\varepsilon >0$ and any given deterministic initial condition, the probability that particular solutions  never blow up is bigger than $1- \varepsilon$.

\item We may wonder if local null-controllability holds for \eqref{eq:heat_seminlinear}, i.e.  if there exists $\delta >0$ such that for every initial data $y_0 \in L^2(\Omega,\fil_0;H^1(\dom))$, $\norme{y_0}_{L^2(\Omega,\fil_0;H_0^1(\dom))} \leq \delta$, one can find a control $h \in L^2_\fil(0,T;L^2(\dom_0))$ such that the solution $y \in L_{\mathcal{F}}^2(\Omega; C([0,T];L^2(\dom)))$ of \eqref{eq:heat_seminlinear} satisfies $y(T, \cdot) = 0$ a.s. This is an interesting open question and new ideas have to be introduced in order to solve this problem.

\item Related to the comments above, we shall emphasize that by our method, the existence and uniqueness of the solution to the uncontrolled equation \eqref{eq:heat_seminlinear} is not a prerequisite for studying its controllability. Actually, it is well-known that without imposing any growth, sign condition, or monotonicity condition on the nonlinear terms, the solutions may not exist globally and blow-up in finite time might occur (see, e.g., the seminal work \cite{Par79}, the newer references \cite{Zha09,DKZ19}, and the references within for some results and remarks in this direction).
In turn, our method restricts to a truncated case and yields the existence of a solution in the weighted space $X_T$ which by construction implies the controllability constraint $y(T,\cdot)=0$ in $\dom$, a.s. 

\item Among the physical examples that our results cover, let us quote the stochastic Allen-Cahn equation with $f(y) = y - y^3$, up to the change of variable $y \leftarrow e^{t} y$ and the Burgers equation in the one dimensional case, i.e. $f(y, \partial_{x} y) = -y \partial_{x} y$. We refer to \cite{DPD99} where the optimal control of the stochastic Burgers equation is studied. Note also that the multiplicative noise term $g(y) \d W(t)$ can represent the existence of external perturbations or a lack of knowledge of certain physical parameters. Its importance is well-known in physics and biology, see for instance, \cite{KS10,MMQ11,WXZZ16,KY20}.

\item Let us mention an open problem that could be addressed in the future. Consider the stochastic Navier-Stokes equation for $n=2, 3$, 
\begin{equation}\label{eq:navier_stokes}
\begin{cases}
\d y = (\Delta y - y \cdot \nabla y  - \nabla p +  \chi_{\dom_0}h  )\dt+(a y)\d W(t) &\text{in } (0,T)\times\dom, \\
\text{div}\ y = 0&\text{in } (0,T)\times\dom, \\
y=0 &\text{on } (0,T)\times\Gamma, \\
y(0, \cdot)=y_0 &\text{in } \dom.
\end{cases}
\end{equation}
We may wonder if \eqref{eq:navier_stokes} is statistically locally null-controllable? To prove this type of result, a good strategy seems to first prove the null-controllability of the Stokes equation, with one control localized in the drift term, by combining the proofs in \cite{CSL16} and \cite{Lu11}. Then, one could adapt the method present in this paper to deal with the nonlinear term $y \cdot \nabla y$. Difficulties will appear by estimating this nonlinear term due to the fact that $H^1(\dom)$ does not embed in $L^{\infty}(\dom)$ for $n \geq 2$. Probably, one should work in $W^{1, p}(\dom)$, which embeds in $L^{\infty}(\dom)$ for $p>2$.
\end{itemize}

\subsection{Bibliographical comments}

In the deterministic setting, the (small-time) null-controllability of the heat equation has been proved independently in the seminal papers \cite{LR95} and \cite{fursi}. Both proofs rely on Carleman estimates. The local null-controllability of semilinear parabolic equations is also established in \cite[Chapter 4]{fursi} by a linearization argument. Then, in \cite{FCZ00} and \cite{Bar00}, the global null-controllability for slightly superlinear heat equations is obtained.

The study of null-controllability of stochastic linear heat equations was first performed in \cite{BRT03}. In particular, the authors remark that the null-controllability of forward equations was a challenging topic, this is why results have been established in different settings. In \cite{TZ09}, the authors prove the null-controllability of forward parabolic equations by introducing two controls, one localized in the drift term and another in the diffusion term. This result was obtained thanks to Carleman estimate for backward parabolic equations. Then, in \cite{Lu11}, the control in the diffusion term is removed, assuming that the coefficients of the parabolic operator do not depend on the spatial variable. The strategy of obtaining such a result relies on the Lebeau-Robbiano method \cite{LR95} adapted to the stochastic setting.

In the nonlinear setting, the result of null-controllability for semilinear parabolic equations was deemed as a difficult problem even for globally Lipschitz nonlinearities (see \cite[Remark 2.6]{TZ09}), due to the lack of compactness. Nonetheless, in our recent work \cite{HSLBP20} we have overcame this difficulty by presenting a new Carleman estimate and a Banach fixed point procedure, where compactness is not needed. In spite of this new result, the question of how to address the controllability for semilinear equations where the global condition for the nonlinearity is dropped is still open. Due to the fact that maximal regularity arguments for stochastic parabolic equations give only regularity in time and space but not in probability, the nonlinearity is difficult to estimate in suitable spaces in a fixed-point procedure. For this reason, in this paper, we study some local controllability properties.

\subsection{Strategy of proof}
In this part, we explain the proof of \Cref{th:mainresult2} that we split into different main steps. Note that \Cref{th:mainresult1} would be actually a byproduct of the proof of \Cref{th:mainresult2}.
\begin{itemize}
\item First, we linearize the equation \eqref{eq:heat_seminlinear} around $0$ to obtain a stochastic (linear) heat equation. By \cite[Theorem 1.1]{Lu11}, we know that this equation is small-time (globally) null-controllable. Moreover, by working a little bit more, we are able to prove that the cost of null-controllability in time $T>0$, denoted by $C_T$,  behaves as $C_T \leq \exp(C/T) >0 $ where the constant $C > 0$ depends on $\dom$, $\dom_0$ and $a$, see \Cref{sec:controllinear} below.
\item Secondly, we employ an adaptation of the well-known source term method of \cite{LLT13} to the stochastic setting in order to prove the null-controllability of the stochastic linear heat equation with a source term exponentially decreasing as $t \rightarrow T$, see \Cref{sec:stochasticsource} below. We remark that as a byproduct of this stochastic source term method is a new observability estimate for the backward heat equation, see \Cref{sec:byproductnewobs} below.
\item To conclude the proof of \Cref{th:mainresult2}, the application of a  Banach fixed point strategy is not straightforward. Indeed, maximal regularity arguments for stochastic parabolic equations give us only regularity in time and space, but not in probability. This leads to some trouble for estimating the semilinearity $f(y)$. This is why we first replace the semilinearity $f$ by the truncated nonlinearity $f_R$, defined in \eqref{eq:deffR}, for which we are able to perform a Banach fixed point argument for $R> 0$ sufficiently small. All of this actually leads to the proof of \Cref{th:mainresult1}. 
\item The conclusion of \Cref{th:mainresult2} will follow from \Cref{th:mainresult1} and Markov's inequality.
\end{itemize}

\section{Null-controllability result for the linearized stochastic heat equation}\label{sec:control_forward}

\subsection{An estimate of the control cost for the stochastic heat equation}
\label{sec:controllinear}

For a given positive time $\tau >0$, we introduce the notations $Q_{\tau}=\dom\times(0,\tau)$, $\Sigma_{\tau}=\Gamma\times(0,\tau)$.

We linearize \eqref{eq:heat_seminlinear} around $0$ and obtain 
\begin{equation}\label{eq:heat_linear}
\begin{cases}
\d y = (\Delta y + \chi_{\dom_0}h )\dt+(ay)\d W(t) &\text{in } Q_\tau, \\
y=0 &\text{on } \Sigma_\tau, \\
y(0, \cdot)=y_0 &\text{in } \dom.
\end{cases}
\end{equation}

We have the following result.
\begin{prop}\label{prop:controllinear}
For every $\tau >0$, $y_0\in L^2(\Omega,\fil_0;L^2(\dom))$, there exists $h\in L^2_\fil(0,\tau;L^2(\dom_0))$ such that 
$y(\tau)=0$ in $\dom$, a.s. Moreover, we have the following estimate
\begin{equation}
\label{eq:costlinear}
\esp\left(\iint_{\dom_0\times(0,\tau)}|h|^2\dx\dt \right) \leq C_\tau \esp\left(\|y_0\|^2_{L^2(\dom)}\right),
\end{equation}
where $C_\tau=Ce^{C/\tau}$ with a positive constant $C>0$ only depending on $\dom$, $\dom_0$ and $a$.
\end{prop}

This result is actually already known in the literature. It was established in a slightly different framework in \cite[Theorem 1.1]{Lu11} (see also \cite[Theorem 1.1]{LL18} for a more general case of coupled systems). Nonetheless, in such works, the control cost is not made explicit and in their current form the results are not suitable for our purposes. 

Below, we give a description of the main parts needed to achieve the proof of \Cref{prop:controllinear} and we pay special attention at the end in the dependence of $T$ for obtaining the constant $C_T$. For this reason, in what follows, we always assume that $T\in(0,1)$.

\begin{proof}[Proof of \Cref{prop:controllinear}]

 The proof is based on the classical Lebeau-Robbiano strategy introduced in \cite{LR95} and for the sake of presentation we follow the methodology in \cite[Section 3]{LeB18}. We split the proof in three main steps.

\textbf{Step 1: A controllability result for low frequencies}. We consider the unbounded linear operator in $L^2(\dom)$ given by $\left(-\Delta,H^2(\dom)\cap H_0^1(\dom)\right)$. Let $(\lambda_k)_{k\geq 1}$ and $(\phi_k)_{k\geq 1}$ be the corresponding eigenvalues and (normalized) eigenfunctions, i.e., $-\Delta \phi_k=\lambda_k\phi_k$ and $(\phi_k,\phi_{l})=\delta_{k,l}$. It is clear that $(\phi_k)_{k\geq 1}$ is an orthonormal basis of $L^2(\dom)$. For $\lambda>0$, we define the finite dimensional space $E_\lambda=\left\{\sum_{\lambda_k\leq \lambda}c_k\phi_k:c_k\in\R\right\}\subset L^2(\dom)$ and we denote by $\Pi_{E_\lambda}$ the orthogonal projection from $L^2(\dom)$ in $E_\lambda$.

The first part consists in obtaining an observability inequality for the adjoint system
\begin{equation}\label{eq:adj_gen_Bis}
\begin{cases}
\d z = -(\Delta z + \ov{z})\dt+\ov{z} \d{W(t)} &\text{in } Q_{\tau}, \\
z=0 &\text{on } \Sigma_{\tau}, \\
z(\tau, \cdot)=z_{\tau}\in E_{\lambda} &\text{in } \dom.
\end{cases}
\end{equation}
The result is the following.
\begin{lem}\label{prop:obs_finite}
There exists $C>0$ such that for every $\tau\in (0,T)$, $\lambda\geq \lambda_1$, and $z_\tau\in  L^2(\Omega,\mathcal F_{\tau};E_{\lambda})$, the solution $z$ to \eqref{eq:adj_gen_Bis} satisfies
\begin{equation*}
\esp\left(\norme{z(0)}^2
_{L^2(\dom)}\right)\leq \frac{C}{\tau}e^{C\sqrt{\lambda}}\esp\left(\iint_{\dom_0\times(0,\tau)}|z|^2\dx\dt\right).
\end{equation*}
\end{lem}
This result can be proved as in \cite[Proposition 2.1]{Lu11} with only minor modifications, so we omit it. By means of a classical duality argument, \Cref{prop:obs_finite} yields a partial controllability result for the forward system
\begin{equation}\label{eq:app_forward}
\begin{cases}
\d y = (\Delta y + h)\dt+ y\, \d{W(t)} &\text{in } Q_{\tau}, \\
y=0 &\text{on } \Sigma_{\tau}, \\
y(0, \cdot)=y_0 &\text{in } \dom.
\end{cases}
\end{equation}

\begin{lem}\label{prop:app_control_partial}
There exist constants $C,C_2>0$ such that for every $\tau\in(0,T)$, $\lambda\geq \lambda_1$, $y_0\in  L^2(\Omega,\mathcal F_{\tau};L^2(\dom))$, there exists a control $h_\lambda\in L^2_{\mathcal F}(0,T;L^2(\dom_0))$ verifying 
\begin{equation}\label{eq:app_cost_control}
\|h_\lambda\|_{L^2_{\mathcal F}(0,\tau;L^2(\dom_0))}\leq \frac{C}{\tau}e^{C\sqrt{\lambda}}\esp\left(\norme{y_0}^2_{L^2(\dom)}\right),
\end{equation}
such that the corresponding controlled solution $y$ to \eqref{eq:app_forward} satisfies
\begin{equation*}
\Pi_{\lambda}(y(\tau))=0, \quad\textnormal{in $\dom$, \ a.s.,}
\end{equation*}
and 
\begin{equation*}
\esp\left(\norme{y(\tau)}^2_{L^2(\dom)}\right)\leq \left(C_2+\frac{C_2}{\tau}e^{C_2\sqrt{\lambda}}\right)\esp\left(\norme{y_0}^2_{L^2(\dom)}\right).
\end{equation*}
\end{lem}
As in the previous case, \Cref{prop:app_control_partial} can be proved by following \cite[Proposition 2.2]{Lu11} with a few minor adjustments, so we skip the proof. 

 \textbf{Step 2: The Lebeau-Robbiano iterative method}. The second part of the method relies on a time-splitting iterative procedure (see e.g. \cite[Section 6.2]{LeRL12}). Here, we will argue slightly different as compared to  \cite[Section 3]{Lu11} which will allow us to track in a simple way the dependency of the constants with respect to $T$. In the remainder of this section, the constants $C$, $C^\prime$, $C_2$, \ldots, are independent of $T$ and may vary from line to line. 

We split the time interval $[0,T]=\bigcup_{k\in\mathbb N}[a_k,a_{k+1}]$ where $a_k$ is defined recursevely, i.e., $a_0=0$ and $a_{k+1}=a_{k}+2T_k$, where $T_k=T/2^{k+2}$ with $k\in\mathbb N$. Also, for some constant $M>0$ sufficiently large (which will be fixed later on), we define $\mu_k=M2^{2k}$. 

The control strategy can be roughly described as:
\begin{itemize}
\item \textit{Active period.} If $t\in (a_k,a_k+T_k)$, we take the control $h_\lambda$ and the corresponding controlled solution $y$ to \eqref{eq:app_forward} provided by \Cref{prop:app_control_partial} where we select $\lambda=\mu_k$.
\item \textit{Passive period.} If $t\in(a_k+T_k,a_{k+1})$, we set $h\equiv 0$ and use the dissipation properties of the system. 
\end{itemize}

In more detail, during the active period, we take $\lambda=\mu_k$ and by \Cref{prop:app_control_partial} we know that there exists $h_k:=h_{\mu_k}$ such that
\begin{align}
\esp\left(\norme{y(a_k+T_k)}^2_{L^2(\dom)}\right)&\leq \left(C_2+\frac{C_2}{T_k}e^{C_2\sqrt{M}2^k}\right)\esp\left(\norme{y(a_k)}^2_{L^2(\dom)}\right) \notag \\ \label{eq:est_C2}
&\leq \frac{C_2}{T}e^{C_2\sqrt{M}2^k}\esp\left(\norme{y(a_k)}^2_{L^2(\dom)}\right)
\end{align}
and 
\begin{equation}\label{eq:proy_k}
\Pi_{\mu_k}\left(y\left(a_k+T_k\right)\right)=0, \quad a.s.
\end{equation}
In the passive period of control, we will prove that the solution decays exponentially and will provide a suitable bound with an explicit dependency of $T$. This point is different from \cite{Lu11}, where It\^{o}'s formula and a direct computation is performed. Instead, we will use the properties of the heat semigroup $S(t):=e^{t\Delta}$, Burkholder-Davis-Gundy and Gronwall's inequality to deduce the required inequality. 

More precisely, for $t\in (a_k+T_k,a_{k+1})$, $h(t) \equiv 0$, so the solution to \eqref{eq:app_forward} writes as
\begin{equation*}
y(t)=S(t-{a_k}-T_k)y(a_k+T_k)+\int_{a_{k}+T_k}^{t}S(t-s)y(s)\d{W}(s), \quad a.s.
\end{equation*}
Then taking the $L^2$-norm and expectation on both sides, we get 
\begin{align}\notag
&\esp\left(\norme{y(t)}^2_{L^2(\dom)}\right)\\ \notag
&\leq C\esp\left(\norme{S(t-a_{k}-T_k)y(a_k+T_k)}_{L^2(\dom)}^2\right)+C\esp\left(\norme{\int_{a_k+T_k}^{t}S(a_{k+1}-s)y(s)\d{W}(s)}_{L^2(\dom)}^2\right) \\ \label{eq:est_yakm1}
&=: I_1+I_2.
\end{align}
We proceed to estimate $I_1$ and $I_2$. For the first term, using that $\norme{S(t)\psi}_{L^2(\dom)}\leq e^{-\mu t}\norme{\psi}_{L^2(\dom)}$ for all $\psi\in L^2(\dom)$ such that $\Pi_{\mu}(\psi)=0$, we have from \eqref{eq:proy_k}, that
\begin{equation}\label{eq:app_est_I1}
I_1 \leq Ce^{-2\mu_{k}(t-a_k-T_k)}\esp\left(\norme{y(a_k+T_k)}^2_{L^2(\dom)}\right).
\end{equation}
For the second one, using a Burkholder-Davis-Gundy type inequality (see e.g. \cite[Thm. 6.1.2]{LR15}) and the fact that $\norme{S(t)}_{\mathcal L(L^2(\dom))}\leq C$, we obtain
\begin{align}\notag 
I_2&\leq C\esp\left(\sup_{\tau \in[a_k+T_k,t]}\norme{\int_{a_{k+T_k}}^{\tau}S(\tau-s)y(s)\d{W}(s)}_{L^2(\dom)}^2\right) \\ \label{eq:app_est_I2}
&\leq C\int_{a_{k+T_k}}^{t}\esp\left(\norme{S(a_{k+1}-s)y(s)}_{L^2(\dom)}^2\right)\d{s} \leq C\int_{a_{k+T_k}}^{t}\esp\left(\norme{y(s)}_{L^2(\dom)}^2\right)\d{s}.
\end{align}

Hence, using estimates \eqref{eq:app_est_I1}--\eqref{eq:app_est_I2} in \eqref{eq:est_yakm1} and employing Gronwall inequality, we deduce
\begin{align*}
\esp\left(\norme{y(t)}^2_{L^2(\dom)}\right)\leq  Ce^{-2\mu_{k}(t-a_k-T_k)}\esp\left(\norme{y(a_k+T_k)}^2_{L^2(\dom)}\right)(1+e^{C(t-a_k-T_k)}).
\end{align*}
Thus, particularizing the previous estimate with $t=a_{k+1}$, the identity $a_{k+1}=a_{k}+2T_k$ and taking into account that $T\in(0,1)$, we deduce the existence of a constant $C^\prime>0$ independent of $T$ such that
\begin{equation}\label{eq:est_C_prime}
\esp\left(\norme{y(a_{k+1})}^2_{L^2(\dom)}\right)\leq C^\prime e^{-C^\prime M 2^{2k+1}T_k}\esp\left(\norme{y(a_k+T_k)}^2_{L^2(\dom)}\right).
\end{equation}
Noting that $2^{2k+1}T_k=2^kT/2$, we obtain from \eqref{eq:est_C_prime} and \eqref{eq:est_C2}
\begin{equation*}
\esp\left(\norme{y(a_{k+1})}^2_{L^2(\dom)}\right)\leq C^\prime \frac{C_2}{T} e^{-C^\prime M2^k T+C_2\sqrt{M}2^k}\esp\left(\norme{y(a_k)}^2_{L^2(\dom)}\right)
\end{equation*}
whence 
\begin{align}
\esp\left(\norme{y(a_{k+1})}^2_{L^2(\dom)}\right)&\leq \left(\frac{C_2}{T}\right)^{k+1}e^{\sum_{j=0}^k\left(-C^\prime M2^k T+C_2\sqrt{M}2^k\right)}\esp\left(\norme{y_0}^2_{L^2(\dom)}\right) \notag \\ \label{eq:est_final_akm1}
&\leq e^{C_2/T+(C_2\sqrt{M}-C^\prime M T)2^{k+1}}\esp\left(\norme{y_0}^2_{L^2(\dom)}\right).
\end{align}
Once again, taking $M>0$ large enough such that $C_2\sqrt{M}-C^\prime MT<0$ (for instance $M\geq 2(C_2/C^\prime T)^2$), we deduce from \eqref{eq:est_final_akm1} that $\lim_{k\to +\infty}\esp\left(\norme{y(a_k)}_{L^2(\dom)}\right)=0$, which together with \eqref{eq:proy_k} implies $y(T)=0$ in $\dom$, a.s. 

\textbf{Step 3: Conclusion.} We define the control $h$ by gluing all the controls $(h_k)_{k\in\mathbb N}$. Notice that this control is an element of $L^2_{\mathcal F}(0,T;L^2(\dom_0))$. Moreover, we have $$\norme{h}_{L^2_{\mathcal F}(0,T;L^2(\dom_0))}^2=\sum_{k=0}^{+\infty}\norme{h_k}^2_{L^2_{\mathcal F}(a_{k},a_{k}+T_k;L^2(\dom_0))}.$$ Using the estimate on the control \eqref{eq:app_cost_control} on each subinterval $(a_k,a_k+T_k)$ together with \eqref{eq:est_final_akm1}, we get
\begin{equation*}
\norme{h_k}^2_{L^2_{\mathcal F}(a_{k},a_{k}+T_k;L^2(\dom_0))}\leq \frac{C}{T_k}e^{C\sqrt{M}2^k}e^{C/T+(C_2\sqrt{M}-C^\prime M T)2^{k}}\esp\left(\norme{y_0}^2_{L^2(\dom)}\right)
\end{equation*}
for $k\geq 1$ and 
\begin{equation*}
\norme{h_0}^2_{L^2_{\mathcal F}(0,T_0;L^2(\dom_0))} \leq \frac{C}{T_0}e^{C\sqrt{M}}\esp\left(\norme{y_0}^2_{L^2(\dom)}\right).
\end{equation*}
Therefore, using the three above estimates and recalling the definition of $T_k$, we obtain
\begin{equation*}
\norme{h}^2_{L^2_{\mathcal F}(0,T;L^2(\dom_0))}\leq \left(CT^{-1}e^{C\sqrt{M}}+\sum_{k\geq 1}C2^kT^{-1}e^{C/T}e^{(C_2\sqrt{M}-C^\prime M T)2^k}\right)\esp\left(\norme{y_0}^2_{L^2(\dom)}\right).
\end{equation*}
Taking $M$ large enough such that $C_2\sqrt{M}-C^\prime MT/2=-C^{\prime\prime}/T$, with $C^{\prime\prime}>0$, we obtain the the above expression that
\begin{equation*}
\norme{h}^2_{L^2_{\mathcal F}(0,T;L^2(\dom_0))}\leq Ce^{C/T}\int_{0}^{+\infty}\frac{\sigma}{T}e^{-C^{\prime\prime}\frac{\sigma}{T}}\d{\sigma}\,\esp\left(\norme{y_0}^2_{L^2(\dom)}\right)
\leq C {e^{C/T}}\esp\left(\norme{y_0}^2_{L^2(\dom)}\right)
\end{equation*}
which yields the desired result \eqref{eq:costlinear}.
\end{proof}

\subsection{Stochastic source term method}
\label{sec:stochasticsource}

From \Cref{prop:controllinear}, we have an estimate for the control cost $C_{T}$ of the equation \eqref{eq:heat_linear} where $C_{T}=C e^{C/T}$ is defined in \eqref{eq:costlinear}. Then we fix $M>0$ such that $C_{T} \leq M e^{M/T}$ and we introduce the weight
\begin{equation}
\label{eq:defgamma}
\forall t >0,\ \gamma(t) = M e^{M/t}.
\end{equation}
We introduce the notation
\begin{equation*}
s = \min(p,q+1,r) >1,
\end{equation*}
where $p$ and $q$ are defined in \eqref{eq:hypf} and $r$ is defined in \eqref{eq:hypg}.

Let $Q \in (1, \sqrt[s]{2})$ and $P > Q^{s}/(2-Q^{s})$. We define the weights
\begin{align}
\label{eq:rho0}
\forall t \in [0,T),\ \rho_0(t) &:= M^{-P} \exp\left(- \frac{MP}{({Q^{s/2}}-1)(T-t)}\right),\\
\label{eq:rho}
\forall t \in [0,T),\ \rho(t) &:= M^{-1-P} \exp\left(-\frac{(1+P)Q^{s}M}{({Q^{s/2}}-1)(T-t)}\right).
\end{align}

For appropriate source terms $F,G\in L^2_{\fil}(0,T;L^2(\dom))$, we consider
\begin{equation}\label{eq:heat_source}
\begin{cases}
\d y = (\Delta y + \chi_{\dom_0}h + F)\dt+(ay+G)\d W(t) &\text{in } Q_T, \\
y=0 &\text{on } \Sigma_T, \\
y(0, \cdot)=y_0 &\text{in } \dom.
\end{cases}
\end{equation}

We define associated spaces for the source term, the state and the control
\begin{align*}
& \mathcal{S} := \left\{S \in  L^2_{\fil}(0,T;L^2(\dom)) : \frac{S}{\rho} \in  L^2_{\fil}(0,T;L^2(\dom))\right\},\\%\label{eq:defpoidsS}
& \mathcal{Y}:= \left\{y \in L^2_{\fil}(0,T;L^2(\dom)) : \frac{y}{\rho_0} \in L^2_{\fil}(0,T;L^2(\dom))\right\},\\
%\label{eq:defpoidsZ}\\
& \mathcal{H} := \left\{h \in L^2_{\fil}(0,T;L^2(\dom)) : \frac{h}{\rho_{0}}  \in L^2_{\fil}(0,T;L^2(\dom))\right\}.%\label{eq:defpoidsh}
\end{align*}
From the behaviors near $t=T$ of $\rho$ and $\rho_0$, we deduce that each element of $\mathcal{S}$, $\mathcal{Y}$, $\mathcal{H}$ vanishes at $t=T$.

We have the following result.
\begin{prop}\label{prop:source_stochastic}
For every $y_0 \in L^2(\Omega,\fil_0;L^2(\dom))$ and $F,G \in \mathcal{S}$, 
there exists a control $h \in \mathcal{H}$ such that the corresponding controlled solution $y$ to \eqref{eq:heat_linear} belongs to $\mathcal{Y}$. Moreover, there exists a positive constant $C>0$ depending only on $T$, $\dom$, $\dom_0$, $a$, $p$, $q$ and $r$ such that
\begin{align}\notag 
\esp & \left(\sup_{0\leq t\leq T }\left\|\frac{y(t)}{\rho_0(t)}\right\|^2\right)+\esp\left(\int_{0}^{T}\!\!\!\int_{\dom_0}\left|\frac{h}{\rho_0}\right|^2\dx\dt\right) \\ \label{eq:sup_y_limit}
&\leq C\esp\left(\|y_0\|^2_{L^2(\dom)}+\int_{0}^{T} \left[ \left\| \frac{F(t)}{\rho(t)}\right\|^2_{L^2(\dom)}+\left\|\frac{G(t)}{\rho(t)}\right\|^2_{L^2(\dom)}\right]\dt \right).
\end{align}
In particular, since $\rho_0$ is a continuous function satisfying $\rho_0(T)=0$, the above estimate implies 
\begin{equation*}
\label{eq:ytauzero}
y(T)=0 \quad \textnormal{in } \dom, \ \textnormal{a.s.}
\end{equation*}
\end{prop}
\begin{proof}
In the following, the constants $C>0$ can vary from line to line, they are independent of the parameters $k$ and $n$.

For $k\geq 0$, we define $T_k=T-\frac{T}{Q^{ks/2}}$. We easily have the following relation between the weights defined in \eqref{eq:defgamma}, \eqref{eq:rho0} and \eqref{eq:rho}
\begin{equation}
\label{eq:relationweights}
\rho_0(T_{k+2})=\rho(T_{k})\gamma(T_{k+2}-T_{k+1}). 
\end{equation}

For $k \geq 0$, we consider the equation 
\begin{equation}\label{eq:y_1}
\begin{cases}
\d y_1 = (\Delta y_1 + F)\dt+(ay_1+G)\d W(t) &\text{in } (T_{k},T_{k+1})\times \dom, \\
y_1=0 &\text{on } (T_k,T_{k+1})\times \Gamma, \\
y_1(T_k)=0 &\text{in } \dom.
\end{cases}
\end{equation}

We introduce the sequence of random variables $\{a_{k}\}_{k\geq 0}$
\begin{equation*}
a_0 = y_0\ \text{and}\ a_{k+1}=y_1(T_{k+1}).
\end{equation*}
For $k \geq 0$, we also consider the equation
\begin{equation}\label{eq:y_2}
\begin{cases}
\d y_2 = (\Delta y_2 + \chi_{\dom_0}h_k)\dt+(ay_2)\d W(t) &\text{in } (T_{k},T_{k+1})\times \dom, \\
y_2=0 &\text{on } (T_k,T_{k+1})\times \Gamma, \\
y_2(T_k)=a_k &\text{in } \dom.
\end{cases}
\end{equation}
Observe that due to the regularity of the solution of \eqref{eq:y_1}, each $a_{k}$, $k\geq 0$, is $\mathcal F_{T_k}$-measurable and belongs to $L^2(\Omega\times \dom)$. Hence, system \eqref{eq:y_2} is well posed for each $h_k\in L^2_{\mathcal F}(T_k,T_{k+1};L^2(\dom))$ thanks to \Cref{lem:regularity}.

According to \Cref{prop:controllinear}, we can construct a control $h_k\in L^2_{\mathcal F}(T_k,T_{k+1};L^2(\dom))$ such that
\begin{equation*}
y_2(T_{k+1})=0, \quad \textnormal{a.s.}
\end{equation*}
and the following estimates holds
\begin{equation}\label{eq:cost_hk}
\esp\left(\int_{T_k}^{T_{k+1}}\!\!\!\!\int_{\dom_0}|h_k(x,t)|^2\dx\dt\right)\leq \gamma^2(T_{k+1}-T_k)\esp\left(\|a_k\|^2_{L^2(\dom)}\right).
\end{equation}

By \Cref{lem:regularity} applied to \eqref{eq:y_1}, there exists $C_0>0$ such that
\begin{align}\label{eq:ener_y1}
\esp&\left(\|a_{k+1}\|_{L^2(\dom)}^2\right) \leq C_0\esp \left(\int _{T_k}^{T_{k+1}}\left[ \|F(t)\|^2_{L^2(\dom)}+\|G(t)\|^2_{L^2(\dom)}\right]\dt\right) 
\end{align}
By using \eqref{eq:cost_hk}, \eqref{eq:ener_y1}, the fact that $\rho$ is a non-increasing, deterministic function and \eqref{eq:relationweights}, we have
\begin{align*}
&\esp\left(\int_{T_{k+1}}^{T_{k+2}}\!\!\!\int_{\dom_0}|h_{k+1}(x,t)|^2\dx\dt\right)\\
& \leq C_0\gamma^2\left(T_{k+2}-T_{k+1}\right)\esp\left(\int _{T_k}^{T_{k+1}}\left[ \|F(t)\|^2_{L^2(\dom)}+\|G(t)\|^2_{L^2(\dom)}\right]\dt\right)\\
&\leq C_0 \gamma^2\left(T_{k+2}-T_{k+1}\right)\rho^2(T_k)\esp\left(\int_{T_k}^{T_{k+1}}\left[\int_{\dom}\left|\frac{F(t)}{\rho(T_k)}\right|^2+\left|\frac{G(t)}{\rho(T_k)}\right|^2\dx\right]\dt\right) \\
&\leq C_0 \gamma^2\left(T_{k+2}-T_{k+1}\right)\rho^2(T_k)\esp\left(\int_{T_k}^{T_{k+1}}\left[\left\|\frac{F(t)}{\rho(t)}\right\|^2_{L^2(\dom)}+\left\|\frac{G(t)}{\rho(t)}\right\|^2_{L^2(\dom)}\right]\dt\right)\\
& \leq C_0 \rho_0^2\left(T_{k+2}\right) \esp\left(\int_{T_k}^{T_{k+1}}\left[\left\|\frac{F(t)}{\rho(t)}\right\|^2_{L^2(\dom)}+\left\|\frac{G(t)}{\rho(t)}\right\|^2_{L^2(\dom)}\right]\dt\right),
\end{align*}
so we deduce
\begin{align}
\label{eq:hk+1int}
\esp&\left(\int_{T_{k+1}}^{T_{k+2}}\!\!\!\int_{\dom_0}\left|\frac{h_{k+1}(x,t)}{\rho_0(T_{k+2})}\right|^2\dx\dt\right) \leq C_0 \esp\left(\int_{T_k}^{T_{k+1}}\left[\left\|\frac{F(t)}{\rho(t)}\right\|^2_{L^2(\dom)}+\left\|\frac{G(t)}{\rho(t)}\right\|^2_{L^2(\dom)}\right]\dt\right).
\end{align}
Since $\rho_0$ is a non-increasing, deterministic function and using \eqref{eq:hk+1int}, we have
\begin{align}\label{eq:norm_hk_rho}
\esp&\left(\int_{T_{k+1}}^{T_{k+2}}\!\!\!\int_{\dom_0}\left|\frac{h_{k+1}(x,t)}{\rho_0(t)}\right|^2\dx\dt\right) \leq C_0 \esp\left(\int_{T_k}^{T_{k+1}}\left[\left\|\frac{F(t)}{\rho(t)}\right\|^2_{L^2(\dom)}+\left\|\frac{G(t)}{\rho(t)}\right\|^2_{L^2(\dom)}\right]\dt\right).
\end{align}

Let $n\in\mathbb \N^*$. From \eqref{eq:norm_hk_rho}, we have
\begin{align}\notag 
\esp&\left(\int_{T_1}^{T}\!\int_{\dom_0}\sum_{k=0}^{n}\mathbf{1}_{[T_{k+1},T_{k+2})}(t)\left|\frac{h_{k+1}}{\rho_0}\right|^2\dx\dt\right) \\ \label{eq:sum_hk}
&\leq C_0\esp\left(\int_{0}^{T}\sum_{k=0}^n\mathbf{1}_{[T_{k},T_{k+1})}(t)\left[\left\|\frac{F(t)}{\rho(t)}\right\|^2_{L^2(\dom)}+\left\|\frac{G(t)}{\rho(t)}\right\|_{L^2(\dom)}^2\right]\dt\right).
\end{align}
From \eqref{eq:cost_hk} at $k=0$ and recalling that $a_0=y_0$, we have
\begin{equation*}
\esp\left(\int_{0}^{T_1}\!\!\!\int_{\dom_0}\left|h_0\right|^2\dx\dt\right)\leq {\gamma^2\left(T_1\right)}\esp\left(\|y_0\|^2_{L^2(\dom)}\right),
\end{equation*}
so
\begin{equation}\label{eq:est_h0}
\esp\left(\int_{0}^{T_1}\!\!\!\int_{\dom_0}\left|\frac{h_0}{\rho_0}\right|^2\dx\dt\right)\leq \frac{\gamma^2\left(T_1\right)}{\rho_0^2(T_1)}\esp\left(\|y_0\|^2_{L^2(\dom)}\right).
\end{equation}
Putting together \eqref{eq:sum_hk} and \eqref{eq:est_h0} yields the existence of a constant $C>0$ independent of $n$ such that
\begin{align*} 
\esp&\left(\int_{0}^{T_1}\!\!\!\int_{\dom_0}\left|\frac{h_0}{\rho_0}\right|^2\dx\dt\right)+\esp\left(\int_{T_1}^{T}\!\int_{\dom_0}\sum_{k=0}^{n}\mathbf{1}_{[T_{k+1},T_{k+2})}(t)\left|\frac{h_{k+1}}{\rho_0}\right|^2\dx\dt\right) \\
&\leq C\esp\left(\|y_0\|^2_{L^2(\dom)}+\int_{0}^{T}\sum_{k=0}^n\mathbf{1}_{[T_{k},T_{k+1})}(t)\left[\left\|\frac{F(t)}{\rho(t)}\right\|^2_{L^2(\dom)}+\left\|\frac{G(t)}{\rho(t)}\right\|_{L^2(\dom)}^2\right]\dt\right).
\end{align*}
Finally, using Lebesgue's convergence theorem, we can pass to the limit $n\to \infty$ and obtain
\begin{align}\label{est:control_rho_0} 
\esp&\left(\int_{0}^{T}\!\!\!\int_{\dom_0}\left|\frac{h}{\rho_0}\right|^2\dx\dt\right) \leq C\esp\left(\|y_0\|^2_{L^2(\dom)}+\int_{0}^{T}\left[\left\|\frac{F(t)}{\rho(t)}\right\|^2_{L^2(\dom)}+\left\|\frac{G(t)}{\rho(t)}\right\|_{L^2(\dom)}^2\right]\dt\right).
\end{align}
where we have set $h:=\sum_{k=0}^\infty h_k$.

Applying It\^{o}'s rule to $y:=y_1+y_2$ for $t\in[T_k,T_{k+1})$, we get
\begin{align}
\label{eq:y}
\begin{cases}
\d y= (\Delta y+\chi_{\dom_0}h_k+F)\dt+(y+G)\d W(t), & \text{in } (T_k,T_{k+1})\times \dom, \\
y=0 & \text{on }(T_k,T_{k+1})\times\Gamma, \\
y(T_k)=a_k &\text{in }\dom.
\end{cases}
\end{align}
Note that by construction $y$ is continuous at $T_k$ for all $k\geq 0$, therefore by using \eqref{eq:y}, $y$ is a solution to \eqref{eq:heat_linear}.

On the other hand, by \Cref{lem:regularity} applied to \eqref{eq:y}, we have for $k\geq 1$
\begin{align*}
\esp&\left(\sup_{T_k\leq t\leq T_{k+1}} \|y(t)\|_{L^2(\dom)}^2\right) \\ 
&\leq C_0\esp \left(\|a_k\|^2_{L^2(\dom)} +\int _{T_{k}}^{T_{k+1}}\left[ \|\chi_{\dom_0}h_{k}(t)\|^2_{L^2(\dom)}+\|F(t)\|^2_{L^2(\dom)}+\|G(t)\|^2_{L^2(\dom)}\right]\dt\right)
\end{align*}
where we have used that $L^2(\dom)\subset H^{-1}(\dom)$. Then using \eqref{eq:cost_hk} and \eqref{eq:ener_y1} to estimate in the above equation yields
\begin{align*}
\esp&\left(\sup_{T_k\leq t\leq T_{k+1}} \|y(t)\|_{L^2(\dom)}^2\right) \\ 
& \leq C_0\esp \left(\int _{T_{k}}^{T_{k+1}}\left[ \|F(t)\|^2_{L^2(\dom)}+\|G(t)\|^2_{L^2(\dom)}\right]\dt\right) + \left(C_0+\gamma^2(T_{k+1}-T_k)\right)\esp \left( \|a_k\|^2_{L^2(\dom)}\right) \\
& \leq C \gamma^2(T_{k+1}-T_k) \esp \left(\int _{T_{k-1}}^{T_{k+1}}\left[ \|F(t)\|^2_{L^2(\dom)}+\|G(t)\|^2_{L^2(\dom)}\right]\dt\right).
\end{align*}

From identity \eqref{eq:relationweights} we get
\begin{align*}
&\esp\left(\sup_{T_k\leq t\leq T_{k+1}} \|y(t)\|_{L^2(\dom)}^2\right) \\ 
&\leq C \gamma^2(T_{k+1}-T_k) \rho^2(T_{k-1}) \esp \left(\int _{T_{k-1}}^{T_{k+1}}\left[ \left\| \frac{F(t)}{\rho(t)}\right\|^2_{L^2(\dom)}+\left\|\frac{G(t)}{\rho(t)}\right\|^2_{L^2(\dom)}\right]\dt\right)  \\
&\leq C\rho_0^2(T_{k+1}) \esp \left(\int _{T_{k-1}}^{T_{k+1}}\left[ \left\| \frac{F(t)}{\rho(t)}\right\|^2_{L^2(\dom)}+\left\|\frac{G(t)}{\rho(t)}\right\|^2_{L^2(\dom)}\right]\dt\right),
\end{align*}
so by using that $\rho_0$ is non-increasing, we have
\begin{align}
\esp\left(\sup_{T_k\leq t\leq T_{k+1}} \left\|\frac{y(t)}{\rho_0(t)}\right\|_{L^2(\dom)}^2\right) \leq C \esp \left(\int _{T_{k-1}}^{T_{k+1}}\left[ \left\| \frac{F(t)}{\rho(t)}\right\|^2_{L^2(\dom)}+\left\|\frac{G(t)}{\rho(t)}\right\|^2_{L^2(\dom)}\right]\dt\right). \label{eq:sup_y_rho_0}
\end{align}
Moreover, arguing as before, it is not difficult to establish that 
\begin{equation}
\label{eq:estimateynear0}
\esp  \left(\sup_{0\leq t\leq T_1}\left\|\frac{y(t)}{\rho_0(t)}\right\|^2\right) \leq C \left(\|y_0\|^2_{L^2(\dom)}+\int_{T_0}^{T_1} \left[ \left\| \frac{F(t)}{\rho(t)}\right\|^2_{L^2(\dom)}+\left\|\frac{G(t)}{\rho(t)}\right\|^2_{L^2(\dom)}\right]\dt \right).
\end{equation}

Let $n\in\mathbb N$. From \eqref{eq:sup_y_rho_0} and \eqref{eq:estimateynear0}, we have
\begin{align}\notag 
\esp & \left(\sup_{0\leq t\leq T_1}\left\|\frac{y(t)}{\rho_0(t)}\right\|^2\right)+\sum_{k=1}^{n}\esp\left(\sup_{T_{k}\leq t\leq T_{k+1}}\left\|\frac{y(t)}{\rho(t)}\right\|^2\right) \\ \label{eq:sup_y_n}
&\leq \widetilde C\esp\left(\|y_0\|^2_{L^2(\dom)}+\sum_{k=1}^{n}\int_{0}^{T}\mathbf{1}_{[T_{k-1},T_{k+1})} \left[ \left\| \frac{F(t)}{\rho(t)}\right\|^2_{L^2(\dom)}+\left\|\frac{G(t)}{\rho(t)}\right\|^2_{L^2(\dom)}\right]\dt \right)
\end{align}
where $\widetilde C>0$ is uniform with respect to $n$. Letting $n\to \infty$ in \eqref{eq:sup_y_n} yields
\begin{align}
\esp  \left(\sup_{0\leq t\leq T }\left\|\frac{y(t)}{\rho_0(t)}\right\|^2\right) \label{eq:sup_y_limit_aux}
\leq \widetilde C\esp\left(\|y_0\|^2_{L^2(\dom)}+\int_{0}^{T} \left[ \left\| \frac{F(t)}{\rho(t)}\right\|^2_{L^2(\dom)}+\left\|\frac{G(t)}{\rho(t)}\right\|^2_{L^2(\dom)}\right]\dt \right).
\end{align}
Finally, combining \eqref{est:control_rho_0} and \eqref{eq:sup_y_limit_aux} gives the desired result. This concludes the proof.
\end{proof}

\subsection{A byproduct: a new observability estimate for backward parabolic equation}
\label{sec:byproductnewobs}

We introduce the backward parabolic equation
\begin{equation}\label{eq:adj_gen}
\begin{cases}
\d z = -(\Delta z + \ov{z} + \tilde{F})\dt+a \ov{z} \d{W(t)} &\text{in } Q_T, \\
z=0 &\text{on } \Sigma_T, \\
z(T, \cdot)=z_T &\text{in } \dom.
\end{cases}
\end{equation}
where $\tilde{F}\in L^2_{\mathcal F}(0,T;L^2(\dom))$ and $z_T\in L^2(\Omega, \mathcal F_T;L^2(\dom))$. 

Under these conditions, by using \cite[Theorem 3.1]{zhou92}, the equation \eqref{eq:adj_gen} admits a unique solution $(z,\ov{z})\in \left[L^2_{\mathcal F}(\Omega; C([0,T];L^2(\dom)))\cap L^2_{\mathcal{F}}(0,T;H_0^1(\dom))\right]\times L^2_{\mathcal F}(0,T;L^2(\dom))$.

From a classical duality argument (see e.g. \cite[Lemma 2.48 \& Theorem 2.44]{Cor07}) and the duality between \eqref{eq:heat_source} and \eqref{eq:adj_gen}, we have as a consequence of the null-controllability result stated in \Cref{prop:source_stochastic} the following observability inequality.
\begin{cor}\label{cor:obs_ineq_adj}
For every $\tilde{F}\in L^2_{\mathcal F}(0,T;L^2(\dom))$ and $z_T\in L^2(\Omega, \mathcal F_T;L^2(\dom))$, the solution $(z,\ov{z})$ to \eqref{eq:adj_gen} satisfies
\begin{align}\notag 
\esp&\left(\int_{\dom}|z(0)|^2\dx\right)+\esp\left(\int_{Q_T}|\rho z|^2 \dx\dt\right)+\esp\left(\int_{Q_T}|\rho \ov{z}|^2\dx\dt\right)  \\ \label{eq:obs_ineq_adj}
& \leq C\esp\left(\int_{Q_{\dom_0}}|\rho_0z|^2\dx\dt+\int_{Q_T}|\rho_0 \tilde{F}|^2\dx\dt\right)
\end{align}
where $C>0$ is the constant appearing in \eqref{eq:sup_y_limit}.
\end{cor}

Estimate \eqref{eq:obs_ineq_adj} looks like the classical observability inequality for the forward stochastic heat equation shown in \cite[Theorem 2.3]{TZ09} proved by means of Carleman estimates. However, our proof is far from Carleman-based strategies and an important difference can be pointed out. Unlike \cite[Eq. (1.6)]{TZ09}, in our estimate the process $\ov{z}$ stays on the left-hand side of the inequality which allows us to consider only one observation term. Although similar estimates with one observation can be obtained, see \cite{Lu11}, the incorporation of $\ov{z}$ on the left-hand side enables us to study more general control problems in the linear setting which are not covered by previous results. Moreover, this will enable us to study some controllability properties for systems with a nonlinear diffusion term. 

\subsection{Regular controlled trajectories}\label{sec:regular}
\indent The next proposition gives more information on the regularity of the controlled trajectory obtained in \Cref{prop:source_stochastic}. We define the weight $\hat{\rho}$ such that $\hat{\rho}(T) = 0$, satisfying the inequalities
\begin{gather}
\label{eq:defrho}
\rho_0 \leq C \hat{\rho},\ \rho \leq C \hat{\rho},\ |\hat{\rho}'| \rho_0 \leq C \hat{\rho}^2, \\ \label{eq:hat_rho_bound}
 {\hat{\rho}^s \leq C \rho}.
\end{gather}

For instance, one can take
\begin{equation*}
\hat{\rho}(t) =  \exp\left(- \frac{M\zeta}{(Q^{s/2}-1)(T-t)}\right),\ \text{with}\ \frac{(1+P)Q^s}{2} < \zeta < P.
\end{equation*}
\begin{prop}\label{prop:SourceTermReg}
For every $y_0 \in L^2(\Omega,\fil_0;H_0^1(\dom))$, $F \in \mathcal{S}$, $G \in  \mathcal{S}$ such that $\nabla G \in \mathcal{S}$, then there exists a control $h \in \mathcal{H}$, such that the solution $y$ of \eqref{eq:heat_linear} satisfies the following estimate
\begin{align}\notag 
 & \esp \left(\sup_{0\leq t\leq T } \left\|\frac{y(t)}{\hat{\rho}(t)}\right\|^2_{H_0^1(\dom)}\right) + \esp\left(\int_{0}^{T}\left\|\frac{y(t)}{\hat{\rho}(t)}\right\|^2_{H^2(\dom)}\dt\right) \\ \label{eq:reg_extra}
&\leq C\esp\left(\|y_0\|^2_{H_0^1(\dom)}+\int_{0}^{T}  \left\| \frac{F(t)}{\rho(t)}\right\|^2_{L^2(\dom)} \dt + \left\| \frac{G(t)}{\rho(t)}\right\|^2_{H^1(\dom)} \dt \right),
\end{align}
where $C$ is a positive constant depending only on $T$, $\dom$, $\dom_0$, $a$, $p$, $q$ and $r$.
\end{prop}
The proof of \Cref{prop:SourceTermReg} is a straightforward adaptation of \cite[Proposition 2.8]{LLT13}. We sketch it briefly. Let us consider a control $h \in \mathcal{H}$ and $y \in \mathcal{Y}$ the corresponding controlled solution provided by \Cref{prop:source_stochastic}. We define $w:=\frac{y}{\hat{\rho}}$ and by means of It\'{o}'s formula we readily deduce that $w$ verifies
\begin{equation*}
\d{w}=\left(\Delta w+\chi_{\dom_0}\frac{h}{\hat{\rho}}+\frac{F}{\hat{\rho}}-\frac{\hat{\rho}^\prime\rho_0}{\hat{\rho}^2}\frac{y}{\rho_0}\right)\dt+\left(aw+\frac{G}{\hat\rho}\right)\d{W}(t) \quad\text{in } Q_T
\end{equation*}
and the conclusion follows from applying the maximal regularity estimate of \Cref{lem:regularity} and using estimates \eqref{eq:defrho}.
\begin{rmk}
\label{rmk:uniquenesstrajectory}
For each $y_0 \in L^2(\Omega,\fil_0;H_0^1(\dom))$, $F \in \mathcal{S}$, $G \in  \mathcal{S}$ such that $\nabla G \in \mathcal{S}$, by classical arguments, see \cite[Proposition 2.9]{LLT13}, we can fix a control $h \in \mathcal{H}$ such that $y$ satisfies \eqref{eq:reg_extra}, by choosing among those the unique minimizer of the functional 
$$ h \mapsto \norme{h}_{\mathcal{H}}^2 + \esp \left(\sup_{0\leq t\leq T } \left\|\frac{y(t)}{\hat{\rho}(t)}\right\|^2_{H_0^1(\dom)}\right) + \esp\left(\int_{0}^{T}\left\|\frac{y(t)}{\hat{\rho}(t)}\right\|^2_{H^2(\dom)}\dt\right).$$
\end{rmk}

\section{The fixed point argument}\label{sec:fixed_point}

The goal of this section is to prove \Cref{th:mainresult1} and \Cref{th:mainresult2}. 

\subsection{Proof of the global null-controllability result for the truncated equation}\label{subsec:fixed}

\begin{proof}[Proof of \Cref{th:mainresult1}]
We split the proof into three main steps:
\begin{itemize}
\item First, we prove some Lipschitz type estimate on $f$,
\item then, we see how the previous estimate translates for $f_R$,
\item finally, we employ a Banach fixed-point argument to prove \Cref{th:mainresult1}.
\end{itemize}

To simplify, we will only treat the case $n=3$, i.e. $f(y) = \alpha y^p$, with $1<p\leq 3$ and $g(y) = 0$. The other cases can be treated in a similar way, see \Cref{rmk:n1n2} below.

The constants that will appear may vary from line to line but are independent of the parameter $R>0$.

\textbf{Step 1: A Lipschitz estimate for $f$.}

Consider $y_1,y_2\in X_T$. The goal of this step is to prove the following estimate
\begin{align}
\norme{\frac{f(y_1)-f(y_2)}{\rho}}_{L^2(\dom)}
 \leq C &\left(\norme{\frac{y_1}{\hat{\rho}}}_{H^1(\dom)}^{p-1} +\norme{\frac{y_2}{\hat{\rho}}}_{H^1(\dom)}^{p-1} \right) \norme{\frac{y_1-y_2}{\hat{\rho}}}_{H^2(\dom)}.\label{eq:est_f}
\end{align}

First, we have by using $|a^p - b^p| \leq C |a-b||(|a|^{p-1} + |b|^{p-1})$, Hölder's estimate and Minkowski's inequality
\begin{align*}
\|\alpha y_1^p-\alpha y_2^p\|_{L^2(\dom)}  &\leq C \norme{y_1-y_2}_{L^{\infty}(\dom)} \norme{|y_1|^{p-1}+|y_2|^{p-1}}_{L^2(\dom)} \\
&\leq C \norme{y_1-y_2}_{L^{\infty}(\dom)} \left( \norme{y_1}_{L^{2(p-1)}(\dom)}^{p-1} + \norme{y_2}_{L^{2(p-1)}(\dom)}^{p-1}\right).
\end{align*}
So, by using $H^2(\dom) \hookrightarrow L^{\infty}(\dom)$ and $H^1(\dom) \hookrightarrow L^6(\dom) \hookrightarrow  L^{2(p-1)}(\dom)$ because $n = 3$ and $p \leq 3$, we deduce 
\begin{align}
\label{eq:est_f_withoutrho}
\norme{\alpha y_1^p-\alpha y_2^p}_{L^2(\dom)}
& \leq C \norme{y_1-y_2}_{H^2(\dom)}\left(\norme{y_1}_{H^1(\dom)}^{p-1}+\|y_2\|_{H^1(\dom)}^{p-1}\right).
\end{align}
Since the weights $\rho$ and $\hat{\rho}$ are $x$-independent, we can incorporate them in \eqref{eq:est_f_withoutrho} to obtain
\begin{align*}
|\rho|\norme{\frac{f(y_1)-f(y_2)}{\rho}}_{L^2(\dom)}
& \leq C |\hat{\rho}|^p \left(\norme{\frac{y_1}{\hat{\rho}}}_{H^1(\dom)}^{p-1} + \norme{\frac{y_2}{\hat{\rho}}}_{H^1(\dom)}^{p-1}  \right) \norme{\frac{y_1-y_2}{\hat{\rho}}}_{H^2(\dom)},
\end{align*}
which leads to \eqref{eq:est_f}, using \eqref{eq:hat_rho_bound}.

\textbf{Step 2: A Lipschitz estimate for $f_R$.}

We borrow some ideas from \cite{LF14} and \cite{Gao17}. we recall that the space $X_t$ is defined in \eqref{eq:defXt}. Without loss of generality, we assume that
\begin{equation}
\label{eq:asswithoutgenerality}
\norme{y_2}_{X_t}\leq \norme{y_1}_{X_t}.
\end{equation}

Using the definition of $f_R$, see \eqref{eq:deffR}, and triangle inequality we have
\begin{align} \label{eq:step21}
\norme{\frac{f_R(y_1)-f_R(y_2)}{\rho}}_{L^2(\dom)}&=\norme{\frac{\varphi_R(y_1)f(y_1)-\varphi_R(y_2)f(y_2)}{\rho}}_{L^2(\dom)} \leq I_1+I_2,
\end{align}
where
\begin{align}
\label{eq:defI1}I_1 &:=  \norme{\frac{\varphi_R(y_1)\left[f(y_1)-f(y_2)\right]}{\rho}}_{L^2(\dom)},\\
\label{eq:defI2}I_2 &= \norme{\frac{\left[\varphi_R(y_1)-\varphi_R(y_2)\right]f(y_2)}{\rho}}_{L^2(\dom)}.
\end{align}
From the definition \eqref{eq:defvarphiR}, \eqref{eq:asswithoutgenerality} and the mean value theorem, we have
\begin{equation}\label{eq:varphi_diff}
\left|\varphi_R(y_1)-\varphi_R(y_2)\right|=(C/R)\left|\norme{y_1}_{X_t}-\norme{y_2}_{X_t}\right|\chi_{\left\{\norme{y_2}_{X_t}\leq 2R\right\}}.
\end{equation}
Thus, using \eqref{eq:varphi_diff} in \eqref{eq:defI2} and the triangle inequality, we get
\begin{equation} \label{eq:est_I2_inter}
I_2\leq (C/R)\norme{y_1-y_2}_{X_t} \norme{\frac{f(y_2)}{\rho}}_{L^2(\dom)} \chi_{\left\{\norme{y_2}_{X_t}\leq 2R\right\}}.
\end{equation}
So from \eqref{eq:est_I2_inter} and $H^1(\dom) \hookrightarrow L^6(\dom) \hookrightarrow  L^{2p}(\dom)$ because $p \leq 3$, we have
\begin{align}
\notag I_2& \leq (C/R) \norme{y_1-y_2}_{X_t} \norme{\frac{y_2}{\hat{\rho}}}_{H^1(\dom)}^{p}  \chi_{\left\{\norme{y_2}_{X_t}\leq 2R\right\}}\\
& \leq C R^{p-1} \norme{y_1-y_2}_{X_t}\chi_{\left\{\norme{y_2}_{X_t}\leq 2R\right\}} \label{eq:est_I2}.
\end{align}

For $I_1$ defined in \eqref{eq:defI1}, we use the definition of $\varphi_R$ in \eqref{eq:defvarphiR}, the estimate \eqref{eq:est_f} established in Step 1 and \eqref{eq:asswithoutgenerality} to get
\begin{align}\notag
I_1&\leq C \left(\norme{\frac{y_1}{\hat{\rho}}}_{H^1(\dom)}^{p-1} +\norme{\frac{y_2}{\hat{\rho}}}_{H^1(\dom)}^{p-1} \right)\norme{\frac{y_1-y_2}{\hat{\rho}}}_{H^2(\dom)} \chi_{\left\{\norme{y_1}_{X_t}\leq 2R\right\}}\notag\\
& \leq C R^{p-1}  \norme{\frac{y_1-y_2}{\hat{\rho}}}_{H^2(\dom)}\chi_{\left\{\norme{y_1}_{X_t}\leq 2R\right\}} .\label{eq:est_I1}
\end{align}

Finally, we combine \eqref{eq:step21}, \eqref{eq:est_I1} and \eqref{eq:est_I2}
\begin{align}\label{eq:trunc_explicit}
\norme{\frac{f_R(y_1)-f_R(y_2)}{\rho}}_{L^2(\dom)} &\leq C R^{p-1}  \left(\norme{y_1-y_2}_{X_t} + \norme{\frac{y_1-y_2}{\hat{\rho}}}_{H^2(\dom)}\right). 
\end{align}

\textbf{Step 3: A Banach fixed-point argument.}

Let $y_0 \in L^2(\Omega,\fil_0;H_0^1(\dom))$.

We introduce the following mapping
\begin{equation*}
\mathcal{N} : F \in \mathcal{S} \mapsto f_R(y) \in \mathcal{S},
\end{equation*}
where $y$ is the solution to \eqref{eq:heat_linear} defined in \Cref{prop:SourceTermReg} and \Cref{rmk:uniquenesstrajectory}.

First, we show that $\mathcal{N}$ is well-defined. By taking the square of \eqref{eq:trunc_explicit} with $y_2=0$, then integrating in time between $0$ and $T$ then taking the expectation, we obtain 
\begin{align}
\esp\left(\int_0^T \norme{\frac{f_R(y(t))}{\rho(t)}}_{L^2(\dom)}^2 \dt\right) &\leq C^2 R^{2(p-1)}\esp\left(T \norme{y}_{X_T}^2 + \norme{y}_{X_T}^2\right)\leq C^2 R^{2(p-1)}\esp \left( \norme{y}_{X_T}^2\right). \label{eq:EstimateNF}
\end{align}
Note that we have used $\norme{\cdot}_{X_t} \leq \norme{\cdot}_{X_T}$ for every $t \in [0,T]$ in \eqref{eq:EstimateNF}. Then using the estimate \eqref{eq:reg_extra}, we have
\begin{equation*}
\esp\left(\int_0^T \norme{\frac{f_R(y(t))}{\rho(t)}}_{L^2(\dom)}^2 \dt\right) \leq C^2 R^{2(p-1)} \esp\left(\norme{y_0}^2_{H_0^1(\dom)}+\int_0^T \norme{\frac{F(t)}{\rho(t)}}_{L^2(\dom)}^2 \dt\right) < + \infty,
\end{equation*}
which translates into $f_R(y) \in \mathcal{S}$.

Secondly, we show that $\mathcal{N}$ is a strictly contraction mapping. By taking the square of \eqref{eq:trunc_explicit} and arguing as in \eqref{eq:EstimateNF}, we obtain 
\begin{equation*}
\esp\left(\int_0^T \norme{\frac{f_R(y_1(t))-f_R(y_2(t))}{\rho(t)}}_{L^2(\dom)}^2 \dt\right) \leq C^2 R^{2(p-1)}\esp\left( \norme{y_1 - y_2}_{X_T}^2 \right),
\end{equation*}
then using the estimate \eqref{eq:reg_extra} with $y_0=0$, we have
\begin{equation*}
\esp\left(\int_0^T \norme{\frac{f_R(y_1)-f_R(y_2)}{\rho}}_{L^2(\dom)}^2\right) \leq C^2 R^{2(p-1)} \ \esp\left(\int_0^T \norme{\frac{F_1-F_2}{\rho}}_{L^2(\dom)}^2\right),
\end{equation*}
which translates into
\begin{equation}
\label{eq:LipschitzN}
\norme{\mathcal{N}(F_1)-\mathcal{N}(F_2)}_{\mathcal{S}} \leq C^2 R^{2(p-1)} \norme{F_1-F_2}_{\mathcal{S}}.
\end{equation}
So taking $R$ such that
\begin{equation}
\label{eq:defRsmall}
C^2 R^{2(p-1)}  < 1,
\end{equation}
we deduce from \eqref{eq:LipschitzN} that $\mathcal{N}$ is a strictly contraction mapping of the Banach space $\mathcal{S}$ so $\mathcal{N}$ admits a unique fixed point $F$. By calling $y$ the trajectory associated to this source term $F$, we remark that $y$ is the solution to \eqref{eq:heat_seminlinear_R}. 

Moreover, we observe from \eqref{eq:reg_extra} and \eqref{eq:EstimateNF} that
\begin{align*}
 \esp\left(\norme{y}_{X_T}^2 \right)
& \leq C^2 \esp\left(\|y_0\|^2_{H_0^1(\dom)}+\int_0^T \norme{\frac{f_R(y(t))}{\rho(t)}}_{L^2(\dom)}^2 \dt \right)\\
& \leq C^2 \left( \esp\|y_0\|^2_{H_0^1(\dom)}  +  R^{2(p-1)} )\esp  \norme{y}_{X_T}^2 \right).
\end{align*}
so taking $R$ sufficiently small if necessary, we can assume that $C^2 R^{2(p-1)}  < 1$, then
\begin{align}
\esp\left(\norme{y}_{X_T}^2\right) \label{eq:reg_extraInTheProof}
\leq C^2 \esp\left(\|y_0\|^2_{H_0^1(\dom)}\right),
\end{align}
which leads to the expected estimate \eqref{eq:estimatenonlinearityPropfR}. This concludes the the proof of \Cref{th:mainresult1}.
\end{proof}

\begin{rmk}
\label{rmk:n1n2}
The case $n=2$, i.e. $f(y) = \alpha y^p$ with $p \in (1, +\infty)$ could be treated as follows. For Step 1, using that $H^2(\dom) \hookrightarrow L^{\infty}(\dom)$ and $H^1(\dom) \hookrightarrow  L^{2(p-1)}(\dom)$, we can obtain \eqref{eq:est_f_withoutrho} so \eqref{eq:est_f} holds. For Step 2, using $H^1(\dom) \hookrightarrow L^{2p}(\dom)$, we can obtain \eqref{eq:est_I2} so \eqref{eq:trunc_explicit} holds.

For the case $n=1$, i.e. $f(y, \partial_x y) = \alpha y^p + \beta y^q \partial_x y$, $g(y) = \gamma y^r$, using $H^1(\dom) \hookrightarrow L^{\infty}(\dom)$, we can prove the following estimates, 
\begin{equation*}
\norme{\beta y_1^q \partial_x y_1 -\beta y_2^q \partial_x y_2}_{L^2(\dom)}
 \leq C \norme{y_1-y_2}_{H^1(\dom)}\left(\norme{y_1}_{H^1(\dom)}^{q}+\|y_2\|_{H^1(\dom)}^{q}\right),
\end{equation*}
\begin{equation*}
\norme{\gamma y_1^r -\gamma y_2^r }_{H^1(\dom)}
 \leq C \norme{y_1-y_2}_{H^1(\dom)}\left(\norme{y_1}_{H^1(\dom)}^{r-1}+\|y_2\|_{H^1(\dom)}^{r-1}\right),
\end{equation*}
then the Lipschitz estimate \eqref{eq:est_f} in Step 1 becomes
\begin{align*}
&\norme{\frac{f(y_1,\partial_x y_1)-f(y_2,\partial_x y_2)}{\rho}}_{L^2(\dom)}+\norme{\frac{g(y_1)-g(y_2)}{\rho}}_{H^1(\dom)}
\\ 
&\leq C \left(\norme{\frac{y_1}{\hat{\rho}}}_{H^1(\dom)}^{q} +  \norme{\frac{y_1}{\hat{\rho}}}_{H^1(\dom)}^{r-1 }+\norme{\frac{y_1}{\hat{\rho}}}_{H^1(\dom)}^{p-1 } \right) \norme{\frac{y_1-y_2}{\hat{\rho}}}_{H^2(\dom)} \\
&\quad + C \left(\norme{\frac{y_2}{\hat{\rho}}}_{H^1(\dom)}^{q} + \norme{\frac{y_2}{\hat{\rho}}}_{H^1(\dom)}^{r-1} +\norme{\frac{y_2}{\hat{\rho}}}_{H^1(\dom)}^{p-1 } \right)\norme{\frac{y_1-y_2}{\hat{\rho}}}_{H^2(\dom)}.
\end{align*}

For obtaining an estimate similar to \eqref{eq:est_I2} for $f$ and $g$, we use again the embedding of $H^1(\dom)$ in $L^{\infty}(\dom)$ to obtain
\begin{align*}
(C/R)\norme{y_1-y_2}_{X_t} \norme{\frac{\beta y_2^q \partial_x y_2}{\rho}}_{L^2(\dom)} \chi_{\left\{\norme{y_2}_{X_t}\leq 2R\right\}} &\leq (C/R) \norme{y_1-y_2}_{X_t} \norme{\frac{y_2}{\hat{\rho}}}_{H^1(\dom)}^{q+1}  \chi_{\left\{\norme{y_2}_{X_t}\leq 2R\right\}}\\
& \leq C R^{q} \norme{y_1-y_2}_{X_t}\chi_{\left\{\norme{y_2}_{X_t}\leq 2R\right\}},
\end{align*}
and
\begin{align*}
(C/R)\norme{y_1-y_2}_{X_t} \norme{\frac{\gamma y_2^r}{\rho}}_{H^1(\dom)} \chi_{\left\{\norme{y_2}_{X_t}\leq 2R\right\}} &\leq (C/R) \norme{y_1-y_2}_{X_t} \norme{\frac{y_2}{\hat{\rho}}}_{H^1(\dom)}^{r}  \chi_{\left\{\norme{y_2}_{X_t}\leq 2R\right\}}\\
& \leq C R^{r-1} \norme{y_1-y_2}_{X_t}\chi_{\left\{\norme{y_2}_{X_t}\leq 2R\right\}},
\end{align*}
so \eqref{eq:trunc_explicit} in Step 2 should be replaced by  
\begin{align*}
&\norme{\frac{f_R(y_1,\partial_x y_1)-f_R(y_2,\partial_x y_2)}{\rho}}_{L^2(\dom)}+\norme{\frac{g_R(y_1)-g_R(y_2)}{\rho}}_{H^1(\dom)} \\
&\quad \leq C(R^q+R^{r-1}+R^{p-1})\left(\norme{y_1-y_2}_{X_t}+\norme{\frac{y_1-y_2}{\hat{\rho}}}_{H_2(\dom)}\right).
\end{align*}
The rest of the proof is analogous.
\end{rmk}

\begin{rmk}
Observe that in the right hand side of \eqref{eq:reg_extra}, we have to estimate the $H^1$-norm of source term in the diffusion while we only have to estimate the $L^2$-norm of source term in the drift. This is why we can prove  \Cref{th:mainresult1} and \Cref{th:mainresult2} only for the case $g(y, \nabla y) = \gamma y^r$ with $\gamma \in \R$, $r \in (1, +\infty)$ and $n=1$.
\end{rmk}

\begin{rmk}\label{rmk:regularity_source}

In the previous proof, see for instance Step 2, by looking at \eqref{eq:est_I2_inter} and \eqref{eq:est_I2}, another possibility might be to estimate as follows 
\begin{align*}
 I_2& \leq (C/R) \norme{y_1-y_2}_{X_t} \norme{\frac{y_2}{\hat{\rho}}}_{H^2(\dom)}^{p}  \chi_{\left\{\norme{y_2}_{X_t}\leq 2R\right\}},
\end{align*}
using $H^2(\dom) \hookrightarrow L^{2p}(\dom)$ which holds for every $p \in [1, +\infty]$ because $n \leq 3$. With this type of estimate, at first glance, it seems that we can treat nonlinearites $\alpha y^p$, for every $p \in (1, +\infty)$ but the problems comes from that we do not have $
\norme{\frac{y}{\hat{\rho}}}_{H^2(\dom)} \leq \norme{y}_{X_t}$, see the definition of the norm $X_t$ in \eqref{eq:defXt} so one cannot obtain \eqref{eq:est_I2} with this strategy.
\end{rmk}

\subsection{Proof of the statistical local null-controllability result}
\label{sec:prooflocal}

We are now in position to prove \Cref{th:mainresult2}.

\begin{proof}[Proof of \Cref{th:mainresult2}]
Let $T>0$ and $\epsilon >0$ be given. Let $R$ as in \Cref{th:mainresult1} and $C^2$ as in equation \eqref{eq:estimatenonlinearityPropfR}. Let us fix $y_0 \in L^2(\Omega,\fil_0;H_0^1(\dom))$ such that
\begin{equation}
\label{eq:defmsalldelta}
\norme{y_0}_{ L^2(\Omega,\fil_0;H_0^1(\dom))} \leq \delta,
\end{equation}
where $\delta >0$ verifying 
\begin{equation}\label{eq:size_delta}
\frac{C^2 \delta^2}{R^2} \leq \epsilon.
\end{equation}
Thanks to \Cref{th:mainresult1}, we know that there exists a control $h \in \mathcal{H}$, such that the solution $y$ of \eqref{eq:heat_seminlinear_R} satisfies $y(T, \cdot) = 0$ a.s. and the estimate \eqref{eq:estimatenonlinearityPropfR} holds. Notice that this result is independent of the size of the initial datum. 

By Markov's inequality, \eqref{eq:estimatenonlinearityPropfR} and \eqref{eq:defmsalldelta}, we have 
\begin{equation*}
\mathbb{P} \left(\norme{y}_{X_T}^2 > R^2 \right) 
\leq \frac{\esp\left(\norme{y}_{X_T}^2\right) }{R^2} 
\leq  \frac{ C^2 \esp\left(\|y_0\|^2_{H_0^1(\dom)}\right)}{R^2} \leq  \frac{C^2 \delta^2}{R^2},
\end{equation*}
so by \eqref{eq:size_delta}, we deduce 
\begin{equation*}
\mathbb{P} \left( \norme{y}_{X_T} \leq  R \right) \geq 1 - \varepsilon.
\end{equation*} 
Using the fact that $\sup_{t \in [0,T]} \norme{\cdot}_{X_t} = \norme{\cdot}_{X_T}$, we {easily} deduce \eqref{eq:probafRf}. This concludes the proof of \Cref{th:mainresult2}.
\end{proof}

\section{Remarks on the case of the backward equation}

The strategy introduced in the previous sections can be used to deal with the controllability of semilinear backward equations. Most of the arguments can be adapted and only minor adjustments are needed. To fix ideas, let us consider the system given by
\begin{equation}\label{eq:back_semilinear}
\begin{cases}
\d{z}=-(\Delta z+\tilde{f}(z)+\sigma \ov{z}+\chi_{\dom_0}h)\dt+\ov{z}\d{W}(t) &\text{in } Q_T, \\
z=0 &\text{in } \Sigma_T, \\
z(T)=z_T &\text{in } \dom,
\end{cases}
\end{equation}
where $z_T$ is a given initial datum, $\tilde{f}$ is a suitable nonlinear function and $\sigma\in\mathbb R$. Notice that the function $\tilde f$ only depends on the variable $z$. This is due to some technical reasons that we shall explain in more detail in \Cref{rmk:non_z_bar}. To simplify, we take $\tilde{f}(z) = z^2$, but other polynomial semilinearities could be considered. 

As for the forward system, the idea is to find a control $h\in L^2_{\mathcal F}(0,T;L^2(\dom))$ such that $z(0,\cdot)=0$, a.s. 
First, we linearize \eqref{eq:back_semilinear} around 0 to obtain
\begin{equation}\label{eq:back_linearized}
\begin{cases}
\d{z}=-(\Delta z+\chi_{\dom_0}h+\sigma\ov{z})\dt+\ov{z}\d{W}(t) &\text{in } Q_T, \\
z=0 &\text{in } \Sigma_T, \\
z(T)=z_T &\text{in } \dom.
\end{cases}
\end{equation}
For each initidal datum $z_T\in L^2_{\mathcal F}(0,T;L^2(\dom))$, system \eqref{eq:back_linearized} admits a unique solution $(z,\ov{z})\in [L^2_{\mathcal F}(\Omega;C([0,T];L^2(\dom)))\cap L^2_{\mathcal F}(0,T;H_0^1(\dom))]\times L^2_{\mathcal F}(0,T;L^2(\dom))$  (see \Cref{lem:regularity_backward}). 

Following \Cref{sec:control_forward}, the first thing to do is to obtain a controllability result for the linear equation \eqref{eq:back_linearized}. By duality, this can be done by obtaining a suitable observability inequality for its adjoint system. In this case, it is not difficult to see that the adjoint is given by
\begin{align*}
\begin{cases}
\d{r}=\Delta r\dt +\sigma r\d{W(t)} &\text{in }Q_T,  \\
r=0 &\text{on }\Sigma_T, \\
r(0)=r_0 &\text{in }\dom,
\end{cases}
\end{align*}
where $r_0\in L^2(\Omega,\mathcal F_0;L^2(\dom))$. The observability inequality for the above system can be deduced by using the Carleman estimate in \cite[Thm. 1.1]{liu14}. In more detail, we have that there exists a constant $C>0$ such that
\begin{align}\label{eq:obs_forw_random}
\esp\left(\int_{\dom}|r(T)|^2\right)\leq C_T \esp \left(\int_0^T\!\!\!\int_{\dom_0}|r|^2\dx\dt\right)
\end{align}
for all $r_0\in L^2(\Omega,\mathcal F_0;L^2(\dom)$. A close inspection to the proof of \cite[Thm. 1.1]{liu14} allows to conclude that the constant $C_T$ is of the form $Ce^{C/T}$ where $C>0$ only depends on $\sigma$. With this, we have the following result.

\begin{theo}
For every $T>0$, $z_T\in L^2(\Omega,\mathcal F_T;L^2(\dom))$, there exists $h\in L^2_{\mathcal F}(0,T;\dom_0)$ such that $y(0)=0$ in $\dom$, a.s. Moreover, we have the following estimate
\begin{equation*}
\esp\left(\iint_{\dom_0\times(0,T)}|h|^2\dx\dt\right)\leq C_T \esp\left(\|z_T\|_{L^2(\dom)}^2\right)
\end{equation*}
where $C_T=Ce^{C/T}$ with $C>0$ only depending on $\dom$, $\dom_0$ and $\sigma$. 
\end{theo}
The proof of this result is classical and it is consequence of \eqref{eq:obs_forw_random}, we refer to \cite[Section 7]{TZ09} (see also \cite[Section 2.4]{hsp20}). 

With this result at hand, the next step is to prove an analogous result to \Cref{prop:source_stochastic}. Note, however, that this time the equation evolves backward in time and the construction of the weights \eqref{eq:rho0}--\eqref{eq:rho} are not longer useful. Let us fix $M>0$ such that $C_T\leq Me^{M/T}$ and introduce the weight 
\begin{equation*}%\label{eq:def_gamma_tilde}
\forall t\in[0,T), \quad \tilde \gamma(t)=Me^{\frac{M}{T-t}}.
\end{equation*}
Observe that this weight blows-up as $t\to T^{-}$. For some parameters $Q\in (1,\sqrt{2})$ and $P>Q^2/(2-Q^2)$, we define the weights
\begin{align*}%\label{eq:def_tilde_rho}
\forall t\in(0,T], \ \tilde \rho(t):=M^{-1-P}\exp\left(-\frac{(1+P)Q^2M}{(Q-1)t}\right)
\end{align*}
and
\begin{align*}%\label{eq:eq:def_tilde_rho_0}
\forall t\in(0,T], \ \tilde \rho_0(t):=M^{-P}\exp\left(-\frac{PM}{(Q-1)t}\right)
\end{align*}

Notice that these weights are very similar to \eqref{eq:rho0}--\eqref{eq:rho}, however this time they are strictly increasing and they vanish as $t\to 0^+$. For an appropriate source term $\tilde F\in L^2_{\fil}(0,T;L^2(\dom))$, we consider
\begin{equation}\label{eq:heat_source_backward}
\begin{cases}
\d z = (-\Delta z + \chi_{\dom_0}h + \tilde F)\dt+\ov{z}\d W(t) &\text{in } Q_T, \\
z=0 &\text{on } \Sigma_T, \\
z(T,\cdot)=z_T &\text{in } \dom.
\end{cases}
\end{equation}

We define associated spaces for the source term, the state and the control as follows
\begin{align*}
& \tilde{\mathcal{S}} := \left\{S \in  L^2_{\fil}(0,T;L^2(\dom))\ ;\ \frac{S}{\tilde \rho} \in  L^2_{\fil}(0,T;L^2(\dom))\right\},\\%\label{eq:defpoidsS}
& \mathcal{\tilde Y}:= \left\{y \in L^2_{\fil}(0,T;L^2(\dom))\ ;\ \frac{y}{\tilde \rho_0} \in L^2_{\fil}(0,T;L^2(\dom))\right\},\\ %\label{eq:defpoidsZ}
& \tilde{\mathcal{H}} := \left\{h \in L^2_{\fil}(0,T;L^2(\dom))\ ;\ \frac{h}{\tilde \rho_{0}}  \in L^2_{\fil}(0,T;L^2(\dom))\right\}.%\label{eq:defpoidsh}
\end{align*}
From the behaviors near $t=0$ of $\tilde \rho$ and $\tilde \rho_0$, we deduce that each element of $\tilde{\mathcal{S}}$, $\tilde{\mathcal{Y}}$, $\tilde{\mathcal{H}}$ vanishes at $t=0$.

We have the following result.
\begin{prop}\label{prop:source_stochastic_backward}
For every $z_T \in L^2(\Omega,\fil_T;L^2(\dom))$ and $\tilde F\in \tilde{\mathcal{S}}$, 
there exists a control $h \in \tilde{\mathcal{H}}$ such that the corresponding controlled solution $y$ to \eqref{eq:heat_source_backward} belongs to $\tilde{\mathcal{Y}}$. Moreover, there exists a positive constant $C>0$ depending only on $T$, $\dom$, $\dom_0$, such that
\begin{align}\notag 
\esp & \left(\sup_{0\leq t\leq T }\left\|\frac{z(t)}{\tilde \rho_0(t)}\right\|^2\right)+\esp\left(\int_{0}^{T}\!\!\!\int_{\dom_0}\left|\frac{h}{\tilde \rho_0}\right|^2\dx\dt\right) \\ \label{eq:sup_y_limit_backward}
&\leq C\esp\left(\|z_T\|^2_{L^2(\dom)}+\int_{0}^{T}  \left\| \frac{\tilde F(t)}{\tilde \rho(t)}\right\|^2_{L^2(\dom)} \dt \right).
\end{align}
In particular, since $\tilde \rho_0$ is a continuous function satisfying $\tilde \rho_0(0)=0$, the above estimate implies 
\begin{equation*}
z(0)=0 \quad \textnormal{in } \dom, \ \textnormal{a.s.}
\end{equation*}
\end{prop}

The proof is similar to the one of \Cref{prop:source_stochastic} and can be adapted just by taking into account the definitions of $\tilde\rho$, $\tilde\rho_0$, the identity $\tilde\rho_0(t)=\tilde\rho\left(Q^2t\right)\tilde \gamma\left(T+(1-Q)t\right)$ and the first regularity estimate of \Cref{lem:regularity_backward}. For brevity, we omit it. 

As in \Cref{sec:regular}, once this result have been established, a more regular controlled trajectory can be obtained. Indeed, defining a weight $\tilde{\hat\rho}$ such that $\tilde{\hat\rho}(0)=0$ verifying 
\begin{equation*}
\label{eq:defrho_hat}
\tilde \rho_0 \leq C \tilde{\hat{\rho}},\ \tilde{\rho} \leq C \tilde{\hat{\rho}},\ |\tilde{\hat{\rho}}^\prime| \tilde{\rho}_0 \leq C \tilde{\hat{\rho}}^2,
\end{equation*}
we can prove the following result by using the maximal regularity estimate in \Cref{lem:regularity_backward}.
\begin{prop}\label{prop:SourceTermReg_backward}
For every $z_T \in L^2(\Omega,\fil_T;H_0^1(\dom))$ and $\tilde F \in \tilde{\mathcal{S}}$, then there exists an unique control $h$ of minimal norm in $\tilde{\mathcal{H}}$, such that the solution $y$ of \eqref{eq:heat_source_backward} satisfies the following estimate
\begin{align}\notag 
 & \esp \left(\sup_{0\leq t\leq T } \left\|\frac{z(t)}{\tilde{\hat{\rho}}(t)}\right\|^2_{H_0^1(\dom)}\right) + \esp\left(\int_{0}^{T}\left\|\frac{z(t)}{\tilde{\hat{\rho}}(t)}\right\|^2_{H^2(\dom)}\dt\right)+\esp\left(\int_0^{T}\norme{\frac{\ov{z}(t)}{\tilde{\hat\rho}(t)}}_{H^1(\dom)}^2\dt\right) \\ \label{eq:reg_extra_backward}
&\leq C\esp\left(\|y_0\|^2_{H_0^1(\dom)}+\int_{0}^{T} \left\| \frac{\tilde{F}(t)}{\tilde{\rho}(t)}\right\|^2_{L^2(\dom)} \dt  \right),
\end{align}
where $C>0$ is only depends on $T$, $\dom$, and $\dom_0$.
\end{prop}

We conclude this section remarking that the analysis of the semilinear case can be carried out as in \Cref{sec:fixed_point}, we just need to adapt the analysis of the truncated nonlinearity. This can be done by considering the space
\begin{align*}
\tilde{X}_t:=\Bigg\{ &z \in [C([0,t];H_0^1(\dom))\cap L^2(0,t;H^2(\dom))] : \\
			     & \sup_{0\leq s\leq t}\norme{\frac{z(s)}{\tilde{\hat\rho}(s)}}_{H_0^1(\dom)}+\left(\int_0^t \norme{\frac{z(s)}{\tilde{\hat\rho}(s)}}_{H^2(\dom)}^2\right)^{1/2} <+\infty\Bigg\}.
\end{align*}
Thus, following the notation of \Cref{sec:fixed_point}, it is not difficult to see that most arguments can be readily adapted. For our example, we can discover that
\begin{equation*}
\norme{\frac{f_R(z_1)-f_R(z_2)}{\tilde \rho}}\leq C R \left(\norme{z_1-z_2}_{\tilde X_t}+\norme{\frac{z_1-z_2}{\tilde{\hat\rho}}}_{H^2}\right)
\end{equation*}
where $R>0$ is small and $f_R(z)=\varphi_R(\|z\|_{\tilde{X}_t})f(z)$, with $\varphi_R$ as in \eqref{eq:defvarphiR}. Of course, this can be generalized to consider some other polynomial nonlinearities but it is not the goal here.

Then, we can do a fixed point argument analogous to Step 3 in \Cref{subsec:fixed} and obtain results for the global controllability case with truncated nonlinearity (cf. \Cref{th:mainresult1}) and the \textit{statistical} local null-controllability case (cf. \Cref{th:mainresult2}) for the semilinear backward system \eqref{eq:back_semilinear}. For brevity, we omit the details.

\begin{rmk}\label{rmk:non_z_bar}
One may wonder why we cannot consider a more general nonlinearity of the form $f(z,\ov{z})=z^2+\ov{z}^2$. Of course, we can change the the space $\tilde X_t$ and include the process $\ov{z}\in L^2(0,T;H^1(\dom))$ in its definition. However, observe that this only adds an $L^2$-in-time estimate for $\ov{z}$, which does not allow us to obtain a nice Lipschitz estimate for the nonlinearity. This is closely related to \Cref{rmk:regularity_source} and the fact that we cannot estimate, for instance, $\norme{\frac{\ov{z}(s)}{\tilde{\rho}(s)}}_{H^1(\dom)}^2\leq \int_{0}^{t}\norme{\frac{\ov{z}(s)}{\tilde{\rho}(s)}}_{H^1(\dom)}^2\d{s}$.
\end{rmk}

\renewcommand{\abstractname}{Acknowledgements}
\begin{abstract}
\end{abstract}
\vspace{-0.5cm}
The work of the first author was supported by the programme ``Estancias posdoctorales por M\'exico'' of CONACyT, Mexico. The second author was supported by the SysNum cluster of excellence University of Bordeaux.

\appendix

\section{Regularity results}

The following result concerns some regularity estimates for forward stochastic parabolic equations. In a slightly more general form they are due to Krylov and Rozovskii \cite{KR77}. We follow the presentation of \cite[Proposition 2.1]{zhou92}.

\begin{lem}\label{lem:regularity}
Let $\tau \in (0,1)$. We have the following energy estimates for \eqref{eq:heat_source} (with $h\equiv 0$).
\begin{itemize}
\item[a)] Assume that $F\in L^2_{\fil}(0,\tau;H^{-1}(\dom))$, $G\in L^2_{\fil}(0,\tau;L^2(\dom))$, and $y_0\in L^2(\Omega,\mathcal F_{\tau};L^2(\dom))$, then \eqref{eq:heat_source} has a unique solution $y\in L^2_{\fil}(0,\tau; H^1_0(\dom))\cap L^2_{\fil}(\Omega; C([0,\tau];L^2(\dom))$. Moreover, there exists a positive constant $C_0$ independent of $\tau$, $F$, $G$ and $y_0$ such that
\begin{align*}\notag
\esp&\left(\sup_{0\leq t\leq \tau} \|y(t)\|_{L^2(\dom)}^2\right)+\esp\left(\int_{0}^{\tau}\|y(t)\|^2_{H^1_0(\dom)}\dt \right) \\
&\leq C_0\esp \left(\|y_0\|^2_{L^2(\dom)} +\int _{0}^{\tau}\left[ \|F(t)\|^2_{H^{-1}(\dom)}+\|G(t)\|^2_{L^2(\dom)}\right]\dt\right) .
\end{align*}
\item[b)] Assume that $F\in L^2_{\fil}(0,\tau;L^2(\dom))$, $G\in L^2_{\fil}(0,\tau;H^1(\dom))$, and $y_0\in L^2(\Omega,\mathcal F_{\tau};H_0^1(\dom))$, then \eqref{eq:heat_source} has a unique solution $y\in L^2_{\fil}(0,\tau; H^2(\dom))\cap L^2_{\fil}(\Omega; C([0,\tau];H_0^1(\dom))$. Moreover, there exists a positive constant $C_0$ independent of $\tau$, $F$, $G$ and $y_0$ such that
\begin{align*}\notag
\esp&\left(\sup_{0\leq t\leq \tau} \|y(t)\|_{H_0^1(\dom)}^2\right)+\esp\left(\int_{0}^{\tau}\|y(t)\|^2_{H^2(\dom)}\dt \right) \\
&\leq C_0\esp \left(\|y_0\|^2_{H_0^1(\dom)} +\int _{0}^{\tau}\left[ \|F(t)\|^2_{L^{2}(\dom)}+\|G(t)\|^2_{H^1(\dom)}\right]\dt\right) .
\end{align*}
\end{itemize}
\end{lem}

In the following result, we present some regularity estimates for backward stochastic parabolic equations.  We refer to \cite[Theorem 3.1]{zhou92} for a more general result.

\begin{lem}\label{lem:regularity_backward}

Let $\tau \in (0,1)$. We have the following energy estimates for \eqref{eq:back_linearized} (with $h\equiv 0$).

\begin{itemize}

\item[a)] Assume that $\tilde F\in L^2_{\fil}(0,\tau;H^{-1}(\dom))$ and $z_{\tau}\in L^2(\Omega,\mathcal F_{\tau};L^2(\dom))$, then \eqref{eq:back_linearized} has a unique solution $ (z,\bar{z})\in [L^2_{\fil}(0,\tau; H^1_0(\dom))\cap L^2_{\fil}(\Omega; C([0,\tau];L^2(\dom)))]\times L^2_\fil(0,T;L^2(\dom))$. Moreover, there exists a positive constant $C_0$ independent of $\tau$, $\tilde F$ and $z_{\tau}$ such that
\begin{align*}\notag
\esp&\left(\sup_{0\leq t\leq \tau} \|z(t)\|_{L^2(\dom)}^2\right)+\esp\left(\int_{0}^{\tau}\left[\|z(t)\|^2_{H^1_0(\dom)}+\|\bar{z}(t)\|^2_{L^2(\dom)}\right]\dt \right) \\
&\leq C_0\esp \left(\|z_{\tau}\|^2_{L^2(\dom)} +\int _{0}^{\tau}\|\tilde F(t)\|^2_{H^{-1}(\dom)}\dt\right) .
\end{align*}
\item[b)] Assume that $\tilde F\in L^2_{\fil}(0,\tau;L^2(\dom))$ and $z_{\tau}\in L^2(\Omega,\mathcal F_{\tau};H_0^1(\dom))$, then \eqref{eq:back_linearized} has a unique solution $(z,\bar{z})\in  [L^2_{\fil}(0,\tau; H^2(\dom))\cap L^2_{\fil}(\Omega; C([0,\tau];H_0^1(\dom)))]\times L^2_\fil(0,T;H^1(\dom))$. Moreover, there exists a positive constant $C_0$ independent of $\tau$, $\tilde F$ and $z_{\tau}$ such that
\begin{align*}\notag
\esp&\left(\sup_{0\leq t\leq \tau} \|z(t)\|_{H_0^1(\dom)}^2\right)+\esp\left(\int_{0}^{\tau}\left[\|z(t)\|^2_{H^2(\dom)}+\|\bar{z}(t)\|^2_{H^1(\dom)}\right]\dt \right) \\
&\leq C_0\esp \left(\|z_{\tau}\|^2_{H_0^1(\dom)} +\int _{0}^{\tau}\|\tilde F(t)\|^2_{L^2(\dom)}\dt\right) .
\end{align*}
\end{itemize}
\end{lem}

\bibliographystyle{alpha}
\small{\bibliography{bib_nonlin_stoch}}

\begin{thebibliography}{HSLBP20}

\bibitem[Bar00]{Bar00}
Viorel Barbu.
\newblock Exact controllability of the superlinear heat equation.
\newblock {\em Appl. Math. Optim.}, 42(1):73--89, 2000.

\bibitem[BRT03]{BRT03}
Viorel Barbu, Aurel R\u{a}\c{s}canu, and Gianmario Tessitore.
\newblock Carleman estimates and controllability of linear stochastic heat
  equations.
\newblock {\em Appl. Math. Optim.}, 47(2):97--120, 2003.

\bibitem[Cor07]{Cor07}
Jean-Michel Coron.
\newblock {\em Control and nonlinearity}, volume 136 of {\em Mathematical
  Surveys and Monographs}.
\newblock American Mathematical Society, Providence, RI, 2007.

\bibitem[CSL16]{CSL16}
Felipe~W. Chaves-Silva and Gilles Lebeau.
\newblock Spectral inequality and optimal cost of controllability for the
  {S}tokes system.
\newblock {\em ESAIM Control Optim. Calc. Var.}, 22(4):1137--1162, 2016.

\bibitem[DKZ19]{DKZ19}
Robert~C. Dalang, Davar Khoshnevisan, and Tusheng Zhang.
\newblock Global solutions to stochastic reaction-diffusion equations with
  super-linear drift and multiplicative noise.
\newblock {\em Ann. Probab.}, 47(1):519--559, 2019.

\bibitem[DPD99]{DPD99}
Giuseppe Da~Prato and Arnaud Debussche.
\newblock Control of the stochastic {B}urgers model of turbulence.
\newblock {\em SIAM J. Control Optim.}, 37(4):1123--1149, 1999.

\bibitem[FCZ00]{FCZ00}
Enrique Fern\'{a}ndez-Cara and Enrique Zuazua.
\newblock Null and approximate controllability for weakly blowing up semilinear
  heat equations.
\newblock {\em Ann. Inst. H. Poincar\'{e} Anal. Non Lin\'{e}aire},
  17(5):583--616, 2000.

\bibitem[FI96]{fursi}
Andrei~V. Fursikov and Oleg~Yu. Imanuvilov.
\newblock {\em Controllability of evolution equations}, volume~34 of {\em
  Lecture Notes Series}.
\newblock Seoul National University, Research Institute of Mathematics, Global
  Analysis Research Center, Seoul, 1996.

\bibitem[Gao17]{Gao17}
Peng Gao.
\newblock The stochastic {S}wift-{H}ohenberg equation.
\newblock {\em Nonlinearity}, 30(9):3516--3559, 2017.

\bibitem[GHV14]{GHV14}
Nathan~E. Glatt-Holtz and Vlad~C. Vicol.
\newblock Local and global existence of smooth solutions for the stochastic
  {E}uler equations with multiplicative noise.
\newblock {\em Ann. Probab.}, 42(1):80--145, 2014.

\bibitem[HSLBP20]{HSLBP20}
V{\'\i}ctor Hern{\'a}ndez-Santamar{\'\i}a, K{\'e}vin Le~Balc'h, and Liliana
  Peralta.
\newblock Global null-controllability for stochastic semilinear parabolic
  equations.
\newblock {\em In preparation}, 2020.

\bibitem[HSP20]{hsp20}
V{\'\i}ctor Hern{\'a}ndez-Santamar{\'\i}a and Liliana Peralta.
\newblock Controllability results for stochastic coupled systems of fourth-and
  second-order parabolic equations.
\newblock {\em arXiv preprint arXiv:2003.01334}, 2020.

\bibitem[KR77]{KR77}
N.~V. Krylov and B.~L. Rozovski\u{\i}.
\newblock The {C}auchy problem for linear stochastic partial differential
  equations.
\newblock {\em Izv. Akad. Nauk SSSR Ser. Mat.}, 41(6):1329--1347, 1448, 1977.

\bibitem[KS10]{KS10}
Jan Kelkel and Christina Surulescu.
\newblock On a stochastic reaction-diffusion system modeling pattern formation
  on seashells.
\newblock {\em J. Math. Biol.}, 60(6):765--796, 2010.

\bibitem[KY20]{KY20}
Nikos~I. Kavallaris and Yubin Yan.
\newblock Finite-time blow-up of a non-local stochastic parabolic problem.
\newblock {\em Stochastic Process. Appl.}, 130(9):5605--5635, 2020.

\bibitem[L\"11]{Lu11}
Qi~L\"{u}.
\newblock Some results on the controllability of forward stochastic heat
  equations with control on the drift.
\newblock {\em J. Funct. Anal.}, 260(3):832--851, 2011.

\bibitem[LB19]{LeB18}
K{\'e}vin Le~Balc’h.
\newblock Local controllability of reaction-diffusion systems around
  nonnegative stationary states.
\newblock {\em ESAIM Control Optim. Calc. Var.}, 2019.

\bibitem[Lia14]{LF14}
Fei Liang.
\newblock Explosive solutions of stochastic nonlinear beam equations with
  damping.
\newblock {\em J. Math. Anal. Appl.}, 419(2):849--869, 2014.

\bibitem[Liu14]{liu14}
Xu~Liu.
\newblock Global {C}arleman estimate for stochastic parabolic equations, and
  its application.
\newblock {\em ESAIM Control Optim. Calc. Var.}, 20(3):823--839, 2014.

\bibitem[LL18]{LL18}
Lingyang Liu and Xu~Liu.
\newblock Controllability and observability of some coupled stochastic
  parabolic systems.
\newblock {\em Math. Control Relat. Fields}, 8(3-4):829--854, 2018.

\bibitem[LR95]{LR95}
Gilles Lebeau and Luc Robbiano.
\newblock Contr\^{o}le exacte de l'\'{e}quation de la chaleur.
\newblock In {\em S\'{e}minaire sur les \'{E}quations aux {D}\'{e}riv\'{e}es
  {P}artielles, 1994--1995}, pages Exp. No. VII, 13. \'{E}cole Polytech.,
  Palaiseau, 1995.

\bibitem[LR15]{LR15}
Wei Liu and Michael R\"{o}ckner.
\newblock {\em Stochastic partial differential equations: an introduction}.
\newblock Universitext. Springer, Cham, 2015.

\bibitem[LRL12]{LeRL12}
J\'{e}r\^{o}me Le~Rousseau and Gilles Lebeau.
\newblock On {C}arleman estimates for elliptic and parabolic operators.
  {A}pplications to unique continuation and control of parabolic equations.
\newblock {\em ESAIM Control Optim. Calc. Var.}, 18(3):712--747, 2012.

\bibitem[LTT13]{LLT13}
Yuning Liu, Tak\'{e}o Takahashi, and Marius Tucsnak.
\newblock Single input controllability of a simplified fluid-structure
  interaction model.
\newblock {\em ESAIM Control Optim. Calc. Var.}, 19(1):20--42, 2013.

\bibitem[MMQ11]{MMQ11}
Carl Mueller, Leonid Mytnik, and Jeremy Quastel.
\newblock Effect of noise on front propagation in reaction-diffusion equations
  of {KPP} type.
\newblock {\em Invent. Math.}, 184(2):405--453, 2011.

\bibitem[Par79]{Par79}
Etienne Pardoux.
\newblock Stochastic partial differential equations and filtering of diffusion
  processes.
\newblock {\em Stochastics}, 3(2):127--167, 1979.

\bibitem[TZ09]{TZ09}
Shanjian Tang and Xu~Zhang.
\newblock Null controllability for forward and backward stochastic parabolic
  equations.
\newblock {\em SIAM J. Control Optim.}, 48(4):2191--2216, 2009.

\bibitem[WXZZ16]{WXZZ16}
Matthias Winter, Lihu Xu, Jianliang Zhai, and Tusheng Zhang.
\newblock The dynamics of the stochastic shadow {G}ierer-{M}einhardt system.
\newblock {\em J. Differential Equations}, 260(1):84--114, 2016.

\bibitem[Zha09]{Zha09}
Xicheng Zhang.
\newblock On stochastic evolution equations with non-{L}ipschitz coefficients.
\newblock {\em Stoch. Dyn.}, 9(4):549--595, 2009.

\bibitem[Zho92]{zhou92}
Xun~Yu Zhou.
\newblock A duality analysis on stochastic partial differential equations.
\newblock {\em J. Funct. Anal.}, 103(2):275--293, 1992.

\end{thebibliography}

\end{document}